\documentclass[letterpaper]{article}
\usepackage{proceed2e}
\usepackage[margin=1in]{geometry}

\usepackage{times}

\title{Interpreting and using CPDAGs with background knowledge}

\author{ {\bf Emilija Perkovi\'c} \\
Seminar for Statistics \\
ETH Zurich, Switzerland\\
perkovic@stat.math.ethz.ch \\
\And
{\bf Markus Kalisch}\\
Seminar for Statistics \\
ETH Zurich, Switzerland\\
kalisch@stat.math.ethz.ch\\
\And
{\bf Marloes H. Maathuis}\\
Seminar for Statistics \\
ETH Zurich, Switzerland\\
maathuis@stat.math.ethz.ch\\
}

\usepackage{tikz}
\usetikzlibrary{arrows}
\usetikzlibrary{positioning}
\usetikzlibrary{arrows.meta}
\usepackage{enumitem}

\usepackage{amsthm}
\usepackage{amsmath}
\usepackage{amssymb}
\usepackage[authoryear]{natbib}

\usepackage[titlenumbered]{algorithm2e}

\usepackage{tabularx}
\usepackage{color}
\usepackage{setspace}

\usepackage{caption}
\usepackage{subcaption}
\usepackage{graphicx}
\graphicspath{{images/}}
\usepackage{xr}
\usepackage{hyperref}

\newtheorem{theorem}{Theorem}[section]

\newtheorem{lemma}[theorem]{Lemma}
\newtheorem{definition}[theorem]{Definition}
\newtheorem{example}[theorem]{Example}

\theoremstyle{definition}

      \newenvironment{proofof}[1][]{\begin{trivlist}
   \item[\hskip \labelsep {\bfseries Proof of #1.}]}{\hfill{}$\square$\end{trivlist}}

\theoremstyle{remark}
\newtheorem*{remark}{Remark}

\definecolor{blue-violet}{rgb}{0.54, 0.17, 0.89}
\definecolor{antiquefuchsia}{rgb}{0.57, 0.36, 0.51}
\definecolor{amethyst}{rgb}{0.6, 0.4, 0.8}
\definecolor{blue-violet}{rgb}{0.54, 0.17, 0.89}
\definecolor{ao}{rgb}{0.0, 0.5, 0.0}
\definecolor{blue(ncs)}{rgb}{0.0, 0.53, 0.74}
\definecolor{dgreen}{rgb}{0.12, 0.3, 0.17}
\definecolor{cadmiumgreen}{rgb}{0.0, 0.42, 0.24}
\definecolor{darkolivegreen}{rgb}{0.33, 0.42, 0.18}
\definecolor{dartmouthgreen}{rgb}{0.05, 0.5, 0.06}

\newcommand{\tild}{\raise.17ex\hbox{ $\scriptstyle\sim$ }}

\DeclareMathOperator{\PDAG}{PDAG}

\DeclareMathOperator{\CPDAG}{CPDAG}
\DeclareMathOperator{\DAG}{DAG}

\DeclareMathOperator{\PAG}{PAG}
\DeclareMathOperator{\Sib}{Sib}

\DeclareMathOperator{\bPossDe}{b-PossDe}
\DeclareMathOperator{\bPossAn}{b-PossAn}

\DeclareMathOperator{\De}{De}

\DeclareMathOperator{\An}{An}

\DeclareMathOperator{\Pa}{Pa}
\DeclareMathOperator{\Adj}{Adj}

\DeclareMathOperator{\bForbb}{b-Forb}
\DeclareMathOperator{\bAdjust}{b-Adjust}
\DeclareMathOperator{\distancefrom}{distance-from-\!}

\DeclareMathOperator{\ConstructMaxPDAG}{ConstructMaxPDAG}

\newcommand{\forb}{the forbidden set condition}

\newcommand{\bamen}{the b-amenability condition}
\newcommand{\bforb}{the b-forbidden set condition}
\newcommand{\bblck}{the b-blocking condition}

\newcommand{\mpdag}{maximal $\PDAG$}
\newcommand{\Mpdag}{Maximal $\PDAG$}
\newcommand{\MPDAG}{MAXIMAL PDAG}

\newcommand{\vars}[1][V]{\mathbf{#1}}
\newcommand{\e}[1][E]{\mathbf{#1}}

\newcommand{\g}[1][G]{\mathcal{#1}}

\newcommand{\badjustb}[2][X,Y]{\bAdjust(\mathbf{#1},#2)}

\newcommand{\bbf}[2][X,Y]{\bForbb(#1,#2)}
\newcommand{\bfb}[2][X,Y]{\bForbb(\mathbf{#1},#2)}

\newcommand{\pstar}[1][p]{{#1}^{*}}

\makeatletter
\def\ctext#1{\expandafter\@ctext\csname c@#1\endcsname}
\def\@ctext#1{\ifcase#1\or 0\or 1\or 2\or
3\or 4\or 5\fi}
\makeatother
\AddEnumerateCounter{\ctext}{\@ctext}{0}

\makeatletter
\def\cctext#1{\expandafter\@cctext\csname c@#1\endcsname}
\def\@cctext#1{\ifcase#1\or \textbf{b-amenability}\or \textbf{b-forbidden set}\or \textbf{b-blocking}\fi}
\makeatother
\AddEnumerateCounter{\cctext}{\@cctext}{b-amenability}

\makeatletter
\def\btext#1{\expandafter\@btext\csname c@#1\endcsname}
\def\@btext#1{\ifcase#1\or B-i\or B-ii\or B-iii\fi}
\makeatother
\AddEnumerateCounter{\btext}{\@btext}{B-i}

\begin{document}

\maketitle

\begin{abstract}
We develop terminology and methods for working with
maximally oriented partially directed acyclic graphs (\mpdag{}s).
\Mpdag{}s arise from imposing restrictions on a Markov equivalence class of directed acyclic graphs, or equivalently on its graphical representation as a completed partially directed acyclic graph ($\CPDAG$), for example when adding background knowledge about certain edge orientations. Although \mpdag{}s often arise in practice, causal methods have been mostly developed for $\CPDAG$s. In this paper, we extend such methodology to \mpdag{}s. In particular, we develop methodology to read off possible ancestral relationships, we introduce a graphical criterion for covariate adjustment to estimate total causal effects, and we adapt the IDA and joint-IDA frameworks to estimate multi-sets of possible causal effects. We also present a simulation study that illustrates the gain in identifiability of total causal effects as the background knowledge increases. All methods are implemented in the \texttt{R} package \texttt{pcalg}.
\end{abstract}

\section{INTRODUCTION}  \label{sec:intro}

Directed acyclic graphs ($\DAG$s) are used for causal reasoning (e.g., \citealp{Pearl2009}), where directed edges represent direct causal effects. In general, it is impossible to learn a $\DAG$ from (observational) data. Instead, one can learn a completed partially directed acyclic graph ($\CPDAG$) (e.g., \citealp{spirtes2000causation, Chickering02-optimal}), which usually contains some undirected edges and represents a Markov equivalence class of $\DAG$s (see Section \ref{sec:prelim} for definitions). 

Maximally oriented partially directed acyclic graphs (\mpdag{}s) generally contain fewer undirected edges than their corresponding $\CPDAG$s and thus represent fewer $\DAG$s. They arise in various scenarios, for example when adding background knowledge about edge orientations to a $\CPDAG$ \citep{meek1995causal}, when imposing a partial ordering (tiers) of variables before conducting causal structure learning \citep{tetrad1998}, in causal structure  learning from both observational and interventional data \citep{hauserBuehlmann12,wang2017permutation}, or in structure learning under certain model restrictions \citep{hoyer08,ernestroth2016,eigenmann17}.

\Mpdag{}s
appear in the literature under various names, such as interventional essential graphs \citep{hauserBuehlmann12,wang2017permutation}, distribution equivalence patterns \citep{hoyer08}, aggregated $\PDAG$s \citep{eigenmann17}
or simply $\CPDAG$s with background knowledge \citep{meek1995causal}. We refer to them as \mpdag{}s in this paper.

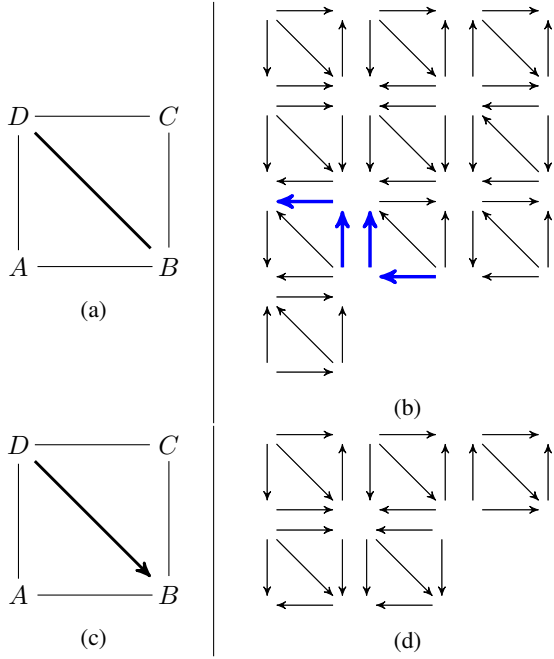
\begin{figure}[!tb]
\tikzstyle{every edge}=[draw,>=stealth',->]
\newcommand\dagvariant[1]{\begin{tikzpicture}[xscale=.5,yscale=0.5]
\node (a) at (0,0) {};
\node (d) at (0,2) {};
\node (b) at (2,0) {};
\node (c) at (2,2) {};
\begin{scope}[gray]
\draw (a) edge [-] (b);
\draw (b) edge [-] (c);
\draw (d) edge [-] (c);
\draw (d) edge [-] (a);
\draw (d) edge [-] (b);
\end{scope}
\draw #1;
\end{tikzpicture}}

\centering
\begin{subfigure}{.4\columnwidth}
  \centering
\begin{tikzpicture}[->,>=latex,shorten >=1pt,auto,node distance=.8cm,scale=1,transform shape]
  \tikzstyle{state}=[inner sep=1pt, minimum size=12pt]

  \node[state] (Xia) at (0,0) {$A$};
  \node[state] (Xka) at (0,2) {$D$};
  \node[state] (Xja) at (2,0) {$B$};
  \node[state] (Xsa) at (2,2) {$C$};

 \draw 		(Xka) edge [line width=1.1pt,-] (Xja);
 \draw 		(Xia) edge  [-] (Xka);
 \draw   	(Xka) edge [-] (Xsa);
 \draw    	(Xsa)  edge [-] (Xja);
 \draw    	(Xja) edge [-] (Xia);
\end{tikzpicture}
  \caption{}
  \label{cpdag11}
\end{subfigure}
\unskip
\vrule
\hspace{.5cm}
\begin{subfigure}{.50\columnwidth}
\dagvariant{
(a) edge [->] (b)
(b) edge [->] (c)
(d) edge [->] (c)
(d) edge [->] (a)
(d) edge [->] (b)
}
\dagvariant{
(b) edge [->] (a)
(b) edge [->] (c)
(d) edge [->] (c)
(d) edge [->] (a)
(d) edge [->] (b)
}
\dagvariant{
(a) edge [->] (b)
(b) edge [->] (c)
(d) edge [->] (c)
(a) edge [->] (d)
(d) edge [->] (b)
}
\dagvariant{
(b) edge [->] (a)
(c) edge [->] (b)
(d) edge [->] (c)
(d) edge [->] (a)
(d) edge [->] (b)
}
\dagvariant{
(b) edge [->] (a)
(c) edge [->] (b)
(c) edge [->] (d)
(d) edge [->] (a)
(d) edge [->] (b)
}
\dagvariant{
(b) edge [->] (a)
(c) edge [->] (b)
(c) edge [->] (d)
(d) edge [->] (a)
(b) edge [->] (d)
}
\dagvariant{
(b) edge [->] (a)
(b) edge [->,line width=1.3pt,color=blue] (c)
(c) edge [->,line width=1.3pt,color=blue] (d)
(d) edge [->] (a)
(b) edge [->] (d)
}
\dagvariant{
(b) edge [->, line width=1.3pt, color=blue] (a)
(b) edge [->] (c)
(d) edge [->] (c)
(a) edge [->, line width=1.3pt, color=blue] (d)
(b) edge [->] (d)
}
\dagvariant{
(b) edge [->] (a)
(b) edge [->] (c)
(d) edge [->] (c)
(d) edge [->] (a)
(b) edge [->] (d)
}
\dagvariant{
(a) edge [->] (b)
(b) edge [->] (c)
(d) edge [->] (c)
(a) edge [->] (d)
(b) edge [->] (d)
}
\caption{}
\label{alldagcpdag}
\end{subfigure}
\\
\begin{subfigure}{.4\columnwidth}
  \centering
\begin{tikzpicture}[->,>=latex,shorten >=1pt,auto,node distance=0.8cm,scale=1,transform shape]
  \tikzstyle{state}=[inner sep=1pt, minimum size=12pt]

  \node[state] (Xia) at (0,0) {$A$};
  \node[state] (Xka) at (0,2) {$D$};
  \node[state] (Xja) at (2,0) {$B$};
  \node[state] (Xsa) at (2,2) {$C$};

  \draw (Xka) edge [->,line width=1.1pt] (Xja);
\draw 		(Xia) edge  [-] (Xka);
 \draw   	(Xka) edge [-] (Xsa);
 \draw    	(Xsa)  edge [-] (Xja);
 \draw    	(Xja) edge [-] (Xia);
\end{tikzpicture}
\caption{}
  \label{mpdag11}
\end{subfigure}
\unskip
\vrule
\hspace{0.5cm}
\begin{subfigure}{.5\columnwidth}
\dagvariant{
(a) edge [->] (b)
(b) edge [->] (c)
(d) edge [->] (c)
(d) edge [->] (a)
(d) edge [->] (b)
}
\dagvariant{
(b) edge [->] (a)
(b) edge [->] (c)
(d) edge [->] (c)
(d) edge [->] (a)
(d) edge [->] (b)
}
\dagvariant{
(a) edge [->] (b)
(b) edge [->] (c)
(d) edge [->] (c)
(a) edge [->] (d)
(d) edge [->] (b)
}
\dagvariant{
(b) edge [->] (a)
(c) edge [->] (b)
(d) edge [->] (c)
(d) edge [->] (a)
(d) edge [->] (b)
}
\dagvariant{
(b) edge [->] (a)
(c) edge [->] (b)
(c) edge [->] (d)
(d) edge [->] (a)
(d) edge [->] (b)
}
\caption{}
\label{alldagmpdag}
\end{subfigure}
\caption{(a) $\CPDAG$ $\g[C]$, (b) all $\DAG$s represented by $\g[C]$, (c) maximally oriented $\PDAG$ $\g$, (d) all $\DAG$s represented by $\g$.}
\label{fig:example_meekgraph11}
\end{figure}

\Mpdag{}s can represent more information about causal relationships than $\CPDAG$s.
However, this additional information has not been fully exploited in practice, since causal methods that are applicable to $\CPDAG$s are not directly applicable to general \mpdag{}s.

We now illustrate the difficulties and differences in working with \mpdag{}s, as opposed to
$\CPDAG$s.
Consider the $\CPDAG$ $\g[C]$ in Figure~\ref{cpdag11}. 
All $\DAG$s represented by $\g[C]$ (see Figure~\ref{alldagcpdag}) have the same adjacencies and unshielded colliders as $\g[C]$.
The undirected edge $A -B$ in $\g[C]$ means that there exists a $\DAG$ represented by $\g[C]$ that contains $A \rightarrow B$, as well as a $\DAG$ that contains $ A \leftarrow B$. Conversely, if $A \rightarrow B$ was in $\g[C]$, then $A \rightarrow B$ would be in all $\DAG$s represented by $\g[C]$.

The paths $B -D$, $B-C-D$ and $B-D$ are possibly directed paths in $\g[C]$, and hence $B$ is a possible ancestor of $D$ in $\g[C]$.
In Figure~\ref{alldagcpdag}, we see that there indeed exist $\DAG$s represented by $\g[C]$ that contain directed paths $B \rightarrow D$, $B \rightarrow C \rightarrow D$, or $B \rightarrow A \rightarrow D$, and $B$ is an ancestor of $D$ in half of the $\DAG$s represented by $\g[C]$.

Now suppose that we know from background knowledge that $D \rightarrow B$.
Adding $D \rightarrow B$ to $\g[C]$ results in the \mpdag{} $\g$ in Figure~\ref{mpdag11}. The $\DAG$s represented by $\g$, shown in Figure~\ref{alldagmpdag}, are exactly the first five $\DAG$s of Figure~\ref{alldagcpdag}. Again, all $\DAG$s represented by $\g$ have the same adjacencies and unshielded colliders as $\g$. Moreover, the interpretation of directed and undirected edges is the same as in the $\CPDAG$. 

In Figure~\ref{mpdag11}, paths $B - C - D$ and $B - A - D$ look like possibly directed paths from $B$ to $D$, so that one could think that $B$ is a possible ancestor of $D$.
However, due to $D \rightarrow B$ and the acyclicity of $\DAG$s, $B$ is not an ancestor of $D$ in any $\DAG$ represented by $\g$. Hence, we need a new way of reading off possible ancestral relationships.

An important difference between \mpdag{}s and $\CPDAG$s, is that the former can contain partially directed cycles while the latter cannot.
In Figure~\ref{mpdag11}, $D \rightarrow B -C-D$ and $D\rightarrow B -D$ are examples of such cycles. As a consequence, the important property of $\CPDAG$s which states that if $D \rightarrow B - C$ is in a $\CPDAG$ $\g[C]$, then $D \rightarrow C$ is also in $\g[C]$ \citep[Lemma 1 from][]{meek1995causal}, does not hold in \mpdag{}s.

Furthermore, let $\g[C]_{undir}$ denote the subgraph of a $\CPDAG$ $\g[C]$ that contains only the undirected edges from $\g[C]$. It is well known that $\g[C]_{undir}$ is triangulated (chordal) and that the connected components of $\g[C]_{undir}$ can be oriented into $\DAG$s without unshielded colliders, independently of each other and the other edges in $\g[C]$, to form all $\DAG$s represented by $\g[C]$ \citep[see the proof of Theorem 4 in][]{meek1995causal}. This result is not valid in \mpdag{}s. To see this, consider the undirected component $\g_{undir}$ of the \mpdag{} $\g$ in Figure~\ref{mpdag11}. Then $\g_{undir}$ is $A -B-C-D-A$ and is not triangulated. Orienting $B-C-D-A$ into a $\DAG$ always leads to at least one unshielded collider. This unshielded collider in $\g_{undir}$ must occur at $A$ or $C$, so that it is shielded in $\g$ by the edge $D\to B$. Hence, one can no longer orient $\g_{undir}$ independently of the directed edges. 

Previous work on \mpdag{}s deals with these issues in two ways, by considering special cases of \mpdag{}s that do not contain partially directed cycles and that retain the same path interpretations as $\CPDAG$s \citep{vanDerZander16}, or by exhaustively searching the space of $\DAG$s represented by the \mpdag{} to find a solution consistent with all these $\DAG$s \citep{hyttinen2015calculus}. 

In this paper, we develop methodology to work with general \mpdag{}s directly.
In \textbf{Section~\ref{sec:prelim}}, we introduce terminology and definitions. In \textbf{Section~\ref{sec:understanding}}, we define possible causal relationships in \mpdag{}s and construct a method to easily read these off from a given \mpdag{}.

In \textbf{Section~\ref{sec:estimate}}, we consider the problem of estimating (possible) total causal effects in \mpdag{}s.
In \textbf{Section~\ref{sec:adjust-mpdag}}, we give a necessary and sufficient graphical criterion for computing total causal effects in \mpdag{}s via covariate adjustment. Our criterion is called the b-adjustment criterion (`b' for background), and builds on \cite{shpitser2012validity,shpitser2012avalidity} and \cite{perkovic15_uai,perkovic16}. We also construct fast algorithms for finding all adjustment sets in \mpdag{}s, using results from \cite{vanconstructing}.
In \textbf{Section~\ref{sec:ida}}, we no longer focus on total causal effects that are identifiable via covariate adjustment, but rather consider computing all possible total causal effects of a node set $\mathbf{X}$ on a node $Y$ in a \mpdag{} $\g$, and collect all these values in a multi-set. Such an approach already exists for $\CPDAG$s under the name (joint-)IDA \citep{MaathuisKalischBuehlmann09,MaathuisColomboKalischBuehlmann10,nandy2014estimating}.  
We develop an efficient semi-local version of (joint-)IDA that is applicable to \mpdag{}s.

In \textbf{Section~\ref{sec:sim}}, we discuss our implementation and present a simulation study where we apply our b-adjustment criterion and semi-local IDA algorithm to $\CPDAG$s with various amounts of added background knowledge. We demonstrate that background knowledge can help considerably in identifying total causal effects via covariate adjustment and in reducing the number of unique values in estimated multi-sets of possible total causal effects. We close with a discussion in \textbf{Section~\ref{sec:disc}}. All proofs can be found in the supplementary material.

\section{PRELIMINARIES} \label{sec:prelim}

\textbf{Nodes, edges and subgraphs.} A graph $\g= (\vars,\e) $ consists of a set of nodes (variables) $ \vars=\left\lbrace X_{1},\dots,X_{p}\right\rbrace$ and a set of edges $ \e $. There is at most one edge between any two nodes, and two nodes are \textit{adjacent} if an edge connects them. 
We call $\rightarrow$ a directed and $-$ a undirected edge. An \textit{induced subgraph} $\g' =(\mathbf{V'}, \mathbf{E'})$ of $\g= (\vars,\e) $ consists of a subset of nodes $\mathbf{V'} \subseteq \mathbf{V}$ and edges $\mathbf{E'} \subseteq \mathbf{E}$ where $\mathbf{E'}$ are all edges in $\mathbf{E}$ between nodes in $\mathbf{V'}$.

\textbf{Paths.} A \textit{path} $p$ from $X$ to $Y$ in $\g$ is a sequence of distinct nodes $\langle X, \dots,Y \rangle$ in which every pair of successive nodes is adjacent. A node $V$ \emph{lies on a path} $p$ if $V$ occurs in the sequence of nodes. If $p = \langle X_1, X_2, \dots , X_k, \rangle, k \ge 2$, then $X_1$ and $X_k$ are \textit{endpoints} of $p$, and any other node $X_i, 1 <i<k,$ is a \textit{non-endpoint} node on $p$. A \textit{directed path} or \textit{causal path} from $X$ to $Y$ is a path from $X$ to $Y$ in which all edges are directed towards $Y$, that is $X \to \dots\to Y$. A \textit{possibly directed path} or \textit{possibly causal path} from $X$ to $Y$ is a path from $X$ to $Y$ that does not contain an edge directed towards $X$. A \textit{non-causal} path from $X$ to $Y$ contains at least one edge directed towards $X$. Throughout, we will refer to a path $p$ from $X$ to $Y$ as non-causal (causal) if and only if it is non-causal (causal) from $X$ to $Y$.
For two disjoint subsets $\mathbf{X}$ and $\mathbf{Y}$ of $\mathbf{V}$, a path from $\mathbf{X}$ to $\mathbf{Y}$
is a path from some $X \in \mathbf{X}$ to some $Y \in \mathbf{Y}$.
A path from $\mathbf{X}$ to $\mathbf{Y}$ is \textit{proper} (w.r.t. $\mathbf{X}$) if only its first node is in $\mathbf{X}$. If $\g$ and $\g^*$ are two graphs with identical adjacencies and $p$ is a path in $\g$, then the \textit{corresponding path} $\pstar$ in $\g^*$ consists of the same node sequence as $p$.

\textbf{Partially directed and directed cycles.}
A directed path from $X$ to $Y$, together with $Y\to X$, forms a \textit{directed cycle}.
A \textit{partially directed cycle} is formed by a possibly directed path from $X$ to $Y$, together with $Y \to X$. 

\textbf{Subsequences and subpaths.} A \textit{subsequence} of a path $p$ is obtained by deleting some nodes from $p$ without changing the order of the remaining nodes. A subsequence of a path is not necessarily a path. For a path $p = \langle X_1,X_2,\dots,X_m \rangle$, the \textit{subpath} from $X_i$ to $X_k$ ($1\le i\le k\le m)$ is the path $p(X_i,X_k) = \langle X_i,X_{i+1},\dots,X_{k}\rangle$.

\textbf{Ancestral relationships.} If $X\to Y$, then $X$ is a \textit{parent} of $Y$. If $X - Y$, then $X$ is a \textit{sibling} of $Y$. If there is a causal path from $X$ to $Y$, then $X$ is an \textit{ancestor} of $Y$, and $Y$ is a \textit{descendant} of $X$. We also use the convention that every node is a descendant and an ancestor of itself.
The sets of parents, siblings, ancestors and descendants of $X$ in~$\g$ are denoted by $\Pa(X,\g)$, $\Sib(X,\g)$, $\An(X,\g)$ and $\De(X,\g)$ respectively. For a set of nodes $\mathbf{X} \subseteq \mathbf{V}$, we let $\Pa(\mathbf{X},\g) = \cup_{X \in \mathbf{X}} \Pa(X,\g)$, with analogous definitions for $\Sib(\mathbf{X},\g)$, $\An(\mathbf{X},\g)$ and $\De(\mathbf{X},\g)$.

\textbf{Colliders, shields and definite status paths.} If a path $p$ contains $X_i \rightarrow X_j \leftarrow X_k$ as a subpath, then $X_j$ is a \textit{collider} on $p$. A path $\langle X_{i},X_{j},X_{k} \rangle$ is an \emph{(un)shielded triple} if $ X_{i} $ and $ X_{k}$ are (not) adjacent. A path is \textit{unshielded} if all successive triples on the path are unshielded. A node $X_{j}$ is a \textit{definite non-collider} on a path $p$ if there is at least one edge out of $X_{j}$ on $p$, or if $X_{i} - X_j - X_k$ is a subpath of $p$ and $\langle X_i,X_j,X_k\rangle$ is an unshielded triple. A node is of \textit{definite status} on a path if it is a collider, a definite non-collider or an endpoint on the path. A path $p$ is of definite status if every node on $p$ is of definite status.

\textbf{D-connection and blocking.} A definite status path \textit{p} from $X$ to $Y$ is \textit{d-connecting} given a node set $\mathbf{Z}$ ($X,Y \notin \mathbf{Z}$) if every definite non-collider on $p$ is not in $\mathbf{Z}$, and every collider on $p$ has a descendant in $\mathbf{Z}$. Otherwise, $\mathbf{Z}$ blocks $p$.

\textbf{DAGs, PDAGs and CPDAGs.}
A \textit{directed graph} contains only directed edges.
A \textit{partially directed graph} may contain both directed and undirected edges.
A directed graph without directed cycles is a \textit{directed acyclic graph $(\DAG)$}.  A \textit{partially directed acyclic graph $(\PDAG)$} is a partially directed graph without directed cycles.

Two disjoint node sets $\mathbf{X}$ and $\mathbf{Y}$ in a $\DAG$ $\g[D]$ are \textit{d-separated} given a node set $\mathbf{Z}$ (pairwise disjoint with $\mathbf{X}$ and $\mathbf{Y}$) if and only if every path from $\mathbf{X}$ to $\mathbf{Y}$ is blocked by $\mathbf{Z}$. Several $\DAG$s can encode the same d-separation relationships. Such $\DAG$s form a \emph{Markov equivalence class} which is uniquely represented  by a \emph{completed partially directed acyclic graph $(\CPDAG)$} \citep{meek1995causal}.
A directed edge $X \rightarrow Y$ in a $\CPDAG$ $\g[C]$ corresponds to $X \rightarrow Y$ in every $\DAG$ in the Markov equivalence class represented by $\g[C]$.
For any undirected edge $X - Y$ in a $\CPDAG$ $\g[C]$, the Markov equivalence class represented by $\g[C]$ contains a $\DAG$ with $X \rightarrow Y$ and a $\DAG$ with $X \leftarrow Y$.

A $\PDAG$ $\g'$ is \textit{represented} by another $\PDAG$ $\g$ (equivalently $\g$ represents $\g'$) if $\g'$ and $\g$ have the same adjacencies and unshielded colliders and every directed edge $X \rightarrow Y$ in $\g$ is also in $\g'$.

\textbf{Maximal PDAGs.} A $\PDAG$ $\g$ is a \textit{maximally oriented} $\PDAG$ if and only if the edge orientations in $\g$ are closed under the orientation rules in Figure~\ref{fig:orientationRules}. Throughout, we will refer to maximally oriented $\PDAG$s as \textit{maximal PDAGs}.

Let $\mathbf{R}$ be a set of required directed edges representing background knowledge. Algorithm~\ref{algo} of \cite{meek1995causal} describes how to incorporate background knowledge $\mathbf{R}$ in a \mpdag{} $\g$. If Algorithm~\ref{algo} does not return a FAIL, then it returns a new \mpdag{} $\g'$ that is represented by $\g$. Background knowledge $\mathbf{R}$ is \textit{consistent} with \mpdag{} $\g$ if and only if Algorithm~\ref{algo} does not return a FAIL \citep{meek1995causal}.
\unskip
\begin{algorithm}[!t]
 \TitleOfAlgo{$\ConstructMaxPDAG$}
 \KwData{\mpdag{} $\g$, background knowledge $\mathbf{R}$}
 \KwResult{\mpdag{} $\g'$ or FAIL}
 Let $\g' = \g$\;
 \While{$\mathbf{R}\neq \emptyset$}{
  Choose an edge $\{X \rightarrow Y\}$ in $\mathbf{R}$ \;
  $\mathbf{R} = \mathbf{R} \setminus \{X \rightarrow Y\}$ \;
  \eIf{$\{X - Y\}$ or $\{X\rightarrow Y\}$ is in $\g'$}{
   Orient $\{X \rightarrow Y\}$ in $\g'$\;
   Close the edge orientations under the rules in Figure~\ref{fig:orientationRules} in $\g'$\;
   }{
   FAIL\;
  }
 }
 \label{algo}
\end{algorithm}

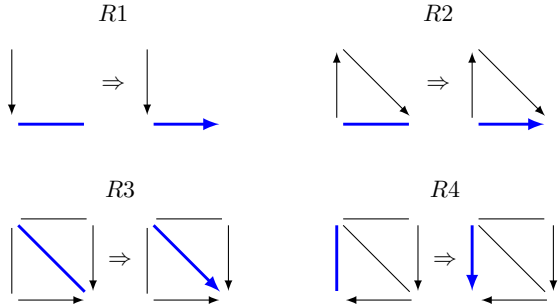
\begin{figure}[!tb]
\vspace{-.2cm}
\centering
\begin{tikzpicture}[->,>=latex,shorten >=1pt,auto,node distance=1.2cm,scale=.9,transform shape]
  \tikzstyle{state}=[inner sep=0.5pt, minimum size=5pt]

  \node[state] (Xia) at (0,0) {};
  \node[state] (Xka) at (0,1.2) {};
  \node[state] (Xja) at (1.2,0) {};

  \path (Xka) edge (Xia);
 \draw[-,line width=1.1pt,blue]
          (Xia) edge (Xja);

  \coordinate [label=above:$R1$] (L1) at (1.5,1.4);
  \coordinate [label=above:$\Rightarrow$] (L2) at (1.5,0.4);

  \node[state] (Xib) at (2,0) {};
  \node[state] (Xkb) at (2,1.2) {};
  \node[state] (Xjb) at (3.2,0) {};

  \path (Xkb) edge (Xib);
   \draw[->,line width=1.1pt,blue]   (Xib) edge (Xjb);

  \node[state] (Xic) at (4.8,0) {};
  \node[state] (Xkc) at (4.8,1.2) {};
  \node[state] (Xjc) at (6,0) {};

  \path (Xic) edge (Xkc)
          (Xkc) edge (Xjc);
   \draw[-,line width=1.1pt,blue]
          (Xic) edge (Xjc);

  \coordinate [label=above:$R2$] (L2) at (6.3,1.4);
  \coordinate [label=above:$\Rightarrow$] (L2) at (6.3,0.4);

  \node[state] (Xid)  at (6.8,0) {};
  \node[state] (Xkd)  at (6.8,1.2) {};
  \node[state] (Xjd)  at (8,0) {};

  \path (Xid) edge (Xkd)
          (Xkd) edge (Xjd);
   \draw[->,line width=1.1pt,blue]  (Xid) edge (Xjd);

  \node[state] (Xie) at (0,-1.4) {};
  \node[state] (Xke) at (1.2,-1.4) {};
  \node[state] (Xle) at (0,-2.6) {};
  \node[state] (Xje) at (1.2,-2.6) {};

  \path (Xke) edge (Xje)
          (Xle) edge (Xje);

  \draw[-,line width=1.1pt,blue]  (Xie) edge (Xje);
  \path[-]
          (Xke) edge (Xie)
          (Xle) edge (Xie);

  \coordinate [label=above:$R3$] (L3) at (1.6,-1.2);
    \coordinate [label=above:$\Rightarrow$] (L3) at (1.6,-2.2);

  \node[state] (Xif) at (2,-1.4) {};
  \node[state] (Xkf) at (3.2,-1.4) {};
  \node[state] (Xlf) at (2,-2.6) {};
  \node[state] (Xjf) at (3.2,-2.6) {};

  \draw[->,line width=1.1pt,blue]  (Xif) edge (Xjf);

  \path (Xkf) edge (Xjf)
          (Xlf) edge (Xjf);
  \path[-]
          (Xkf) edge (Xif)
          (Xlf) edge (Xif);

  \node[state] (Xig) at (4.8,-1.4) {};
  \node[state] (Xjg) at (6,-1.4) {};
  \node[state] (Xkg) at (4.8,-2.6) {};
  \node[state] (Xlg) at (6,-2.6) {};

  \path (Xlg) edge (Xkg)
          (Xjg) edge (Xlg);
  \draw[line width=1.1pt,blue,-]
           (Xig) edge (Xkg);
  \path[-]
          (Xig) edge (Xjg)
          (Xig) edge (Xlg);

  \coordinate [label=above:$R4$] (L4) at (6.4,-1.2);
  \coordinate [label=above:$\Rightarrow$] (L4) at (6.4,-2.2);

  \node[state] (Xih) at (6.8,-1.4) {};
  \node[state] (Xjh) at (8,-1.4) {};
  \node[state] (Xkh) at (6.8,-2.6) {};
  \node[state] (Xlh) at (8,-2.6) {};

  \draw[blue,line width=1.1pt,->]  (Xih) edge (Xkh);
  \path (Xlh) edge (Xkh)
        (Xjh) edge (Xlh);
  \path[-]
          (Xih) edge (Xjh)
           (Xih) edge (Xlh);

\end{tikzpicture}
\caption{The orientation rules from~\cite{meek1995causal}. If the graph on the left-hand side of a rule is an induced subgraph of a $\PDAG$ $\g$, then \textit{orient} the blue undirected edge ({\color{blue} {\bfseries $-$}}) as shown on the right-hand side of the rule.}
\label{fig:orientationRules}
\end{figure}

\textbf{$\g$ and $[\g]$.}
If $\g$ is a \mpdag{}, then $[\g]$ denotes every \mpdag{} represented by $\g$. Thus, $[\g]$ contains all $\DAG$s represented by $\g$, but also all $\PDAG$s (with the same adjacencies and unshielded colliders) that contain more orientations than $\g$. 

\textbf{Causal DAGs and PDAGs.} A density $f$ of $\mathbf{V} = \{X_1,\dots,X_p\}$ is \textit{consistent} with a $\DAG$ $\g[D] =(\mathbf{V},\mathbf{E})$ if it factorizes as $f(\vars)= \prod_{i=1}^{p}f(X_{i}|Pa(X_{i},\g))$ \citep{Pearl2009}. A $\DAG$ is \textit{causal} if every edge $X_{i} \rightarrow X_{j}$ in $\g[D]$ represents a direct causal effect of $X_{i}$ on $X_{j}$ (wrt $\mathbf{V}$). A $\PDAG$ is \textit{causal} if it represents a causal $\DAG$.

We consider interventions $do(\mathbf{X} =\mathbf{x})$ (for $\mathbf{X}\subseteq \mathbf{V}$) or $do(\mathbf{x})$ for shorthand, which represent outside interventions that set $\mathbf{X}$ to $\mathbf{x}$ \citep{Pearl2009}. A density $f$  of $\mathbf{V} = \{X_1,\dots,X_p\}$ is \textit{consistent with a causal DAG} $\g[D] =(\mathbf{V},\mathbf{E})$ if all post-intervention densities $f(\mathbf{v}|do(\mathbf{x}))$ factorize as:
\begin{multline}
f(\mathbf{v}|do(\mathbf{x})) =  \\
=
\begin{cases}
\prod_{X_{i} \in \vars \setminus \mathbf{X}}f(x_{i}|\Pa(x_{i},\g[D])), &  \text{if }\mathbf{X} =\mathbf{x}, \\
0, & \text{otherwise.}
\end{cases}
\label{eq11}
\end{multline}
Equation \eqref{eq11} is known as the truncated factorization formula \citep{Pearl2009}, manipulated density formula \citep{spirtes2000causation} or the g-formula \citep{robins1986new}.
A density $f$ is \textit{consistent with a causal PDAG} $\g$ if it is consistent with a causal $\DAG$ in $[\g]$.
\section{POSSIBLY CAUSAL RELATIONSHIPS IN \MPDAG{}s} \label{sec:understanding}
A basic task for causal reasoning based on a \mpdag{} $\g$ is reading off possible causal relationships between nodes directly from $\g$. The alternative is of course to list all $\DAG$s in $[\g]$ and to read off the causal relationship in each $\DAG$. If $\g$ has many undirected edges, this task quickly becomes a computational burden. We show that possible causal relationships can be read off from the \mpdag{} directly if we modify the definitions of possibly causal and non-causal paths.

In Definition~\ref{def:posdir-path}, we define a \textit{b-possibly causal} path and its complement the \textit{b-non-causal} path for \mpdag{}s. Throughout, we use the prefix \textit{b-} (for background) to distinguish from the analogous terms for $\CPDAG$s. Lemma~\ref{lemma:posdir-path} demonstrates the soundness of Definition~\ref{def:posdir-path}.
\begin{definition}(\textbf{b-possibly causal path, b-non-causal path})
Let $p = \langle X=V_0, \dots , V_k=Y \rangle$, $k \geq 1$ be a path from node $X$ to node $Y$ in a \mpdag{} $\g$. We say that $p$ is b-possibly causal in $\g$ if and only if no edge $V_{i} \leftarrow V_{j}, 0 \le i < j \le k$ is in $\g$. Otherwise, we say that $p$ is b-non-causal path in $\g$.
\label{def:posdir-path}
\end{definition}
\begin{remark} Note that Definition~\ref{def:posdir-path} involves edges that are not on $p$.
\end{remark}
\begin{lemma}
Let $\pstar$ be a path from $X$ to $Y$ in a \mpdag{} $\g$. If $\pstar$ is b-non-causal in $\g$, then for every $\DAG$ $\g[D]$ in $[\g]$  the corresponding path in $\g[D]$ is non-causal. By contraposition, if $p$ is a causal path in at least one $\DAG$ $\g[D]$ in $[\g]$, then the corresponding path in $\g$ is b-possibly causal.
\label{lemma:posdir-path}
\end{lemma}
\noindent{}Note that since every $\CPDAG$ is also a \mpdag{}, all results in this paper subsume existing results for $\CPDAG$s.
We now define \textit{b-possible descendants} (\textit{b-possible ancestors}) in a \mpdag{}.
\begin{definition}$(\bPossDe(\mathbf{X},\g) \text{ and }\bPossAn(\mathbf{X},\g))$
Let $\textbf{X}$ be a node set in a \mpdag{} $\g$. Then $W$ is a b-possible descendant (b-possible ancestor) of $\mathbf{X}$ in $\g$, and we write $W \in \bPossDe(\mathbf{X},\g)$ $(W \in \bPossAn(\mathbf{X},\g))$ if and only if $W \in \mathbf{X}$ or there is a node $X \in \mathbf{X}$ and a b-possibly causal path from $X$ to $W$ ($W$ to $X$) in $\g$.
\label{def:posde-1}
\end{definition}

\begin{example}
We use this example to illustrate b-possibly causal paths and b-possible descendants.
Consider $\CPDAG$ $\g[C]$ in Figure~\ref{cpdag11} and \mpdag{} $\g$ in Figure~\ref{mpdag11}.
We see that $B - C - D$ is a b-possibly causal path in $\g[C]$, and a b-non-causal path in $\g$ due to $D \rightarrow B$ in $\g$.
Conversely, $D-C-B$ is a b-possibly causal path in both $\g[C]$ and $\g$.

Furthermore, $\bPossDe(B,\g[C]) = \{ A,B, C,D\}$ and $\bPossDe(B, \g) = \{A,B,C \}$.
The b-possible descendants of $A,C$ and $D$ nodes are the same in $\g$ and $\g[C]$.
\label{ex:allpossdec}
\end{example}
\subsection{EFFICIENTLY FINDING ALL B-POSSIBLE DESCENDANTS/ANCESTORS}
The b-possibly causal paths present an elegant extension of the notion of possibly causal paths from $\CPDAG$s to \mpdag{}s. Nevertheless, finding all b-possible descendants by checking the b-possibly causal status of a path is non trivial, since it involves considering many edges not on the path. This is cumbersome if the graphs we are dealing with are large and/or dense.  

To solve this issue, we use Lemma~\ref{lemma:unshielded-analog} (analogous to \citealp[Lemma~B.1 in][]{zhang2008completeness}) to only consider paths which are unshielded and hence, of definite status. We then only need consider the edges which are on the definite status paths (Lemma~\ref{lemma:def-stat-posdir}).
Thus, the task of finding all b-possible descendants (ancestors) of $\mathbf{X}$ in a \mpdag{} $\g = (\mathbf{V},\mathbf{E})$ can be done using a depth first search algorithm with computational complexity $O(|\mathbf{V}| + |\mathbf{E}|)$.
\begin{lemma}
Let $\pstar = \langle V_1, \dots , V_k \rangle$ be a definite status path in a \mpdag{} $\g$. Then $\pstar$ is b-possibly causal if and only if there is no $V_i \leftarrow V_{i+1}$, for $i \in \{1,\dots, k-1\}$  in $\g$.
\label{lemma:def-stat-posdir}
\end{lemma}
\begin{lemma}
 Let $X$ and $Y$ be distinct nodes in a \mpdag{} $\g$. If $p$ is s a b-possibly causal path from $X$ to $Y$ in $\g$, then a subsequence $\pstar$ of $p$ forms a b-possibly causal unshielded path from $X$ to $Y$ in~$\g$.
\label{lemma:unshielded-analog}
\end{lemma}

\section{ESTIMATING TOTAL CAUSAL EFFECTS WITH \MPDAG{}S} \label{sec:estimate}
After inferring the existence of a possibly non-zero total causal effect of $\mathbf{X}$ on $\mathbf{Y}$ in a \mpdag{} $\g$ by finding a b-possible causal path from $X$ to $Y$, the natural question to ask is, how big is this effect? How do we calculate this effect using observational data?
Throughout, let $\g=(\vars,\e)$ represent a causal \mpdag{}, and let $\mathbf{X}$, $\mathbf{Y}$ and $\mathbf{Z}$ be pairwise disjoint subsets of $\vars$.
We are interested in the total causal effect of $\mathbf{X}$ on $\mathbf{Y}$.

\subsection{ADJUSTMENT IN \MPDAG{}s} \label{sec:adjust-mpdag}
The most commonly used tool for estimating total causal effects from observational data is covariate adjustment.
We first define the concept of a adjustment sets for \mpdag{}s.
\begin{definition}{(\textbf{Adjustment set})}
   Let $\mathbf{X,Y}$ and $\mathbf{Z}$ be pairwise disjoint node sets in a causal \mpdag{} $\g$. Then $\mathbf{Z}$ is an adjustment set relative to $(\mathbf{X,Y})$ in~$\g$ if for any density $f$ consistent with $\g$ we have
   \begin{equation}
   f(\mathbf{y}|do(\mathbf{x}))=
   \begin{cases}
   f(\mathbf{y}|\mathbf{x}) & \text{if }\mathbf{Z} = \emptyset,\\
   \int_{\mathbf{z}}f(\mathbf{y}|\mathbf{x,z})f(\mathbf{z})d\mathbf{z}  & \text{otherwise.}
   \end{cases}
   \nonumber
   \end{equation}
   \label{defadjustmentmpdag}
\end{definition}\vspace{-0.3cm}
\noindent{}Thus, if $\mathbf{Z}$ is an adjustment set relative to $(\mathbf{X,Y})$ in $\g$, we do not need to find the true underlying causal $\DAG$ in order to find the post-intervention density $f(\mathbf{y}|do(\mathbf{x}))$.

Under the assumption that the density consistent with the underlying causal $\DAG$ can be generated by a linear structural equation model (SEM) with additive noise \citep[see Ch.5 in][for definition of SEMs]{Pearl2009}, covariate adjustment allows the researcher to estimate the total causal effect by performing one multiple linear regression \citep[see][for the non-Gaussian noise result]{nandy2014estimating}. In this setting, if $\mathbf{Z}$ is an adjustment set relative to some nodes $X$, $Y$ in $\g$, then the coefficient of $X$ in the linear regression of $Y$ on $X$ and $\mathbf{Z}$ is the total causal effect of $X$ on $Y$.

There has been recent work on finding graphical criteria for sets that satisfy Definition~\ref{defadjustmentmpdag} relative to some $(\mathbf{X},\mathbf{Y})$. The first such criterion was the back-door criterion \citep{Pearl2009}, which is sound but not complete for adjustment (wrt Definition~\ref{defadjustmentmpdag}). The sound and complete adjustment criterion for $\DAG$s was introduced in \cite{shpitser2012validity,shpitser2012avalidity} and extended to more general graph classes in \cite{perkovic15_uai,perkovic16}. Furthermore, \cite{vanDerZander16} presented an adjustment criterion that can be used in \mpdag{}s that do not contain partially directed cycles. None of these criteria is applicable to general \mpdag{}s.

We present our b-adjustment criterion in Definition~\ref{def:gac-pdag}, building on the results of \cite{shpitser2012validity,shpitser2012avalidity} and \cite{perkovic15_uai,perkovic16}. Our criterion, as well as other results presented in this section, is phrased in the same way as the results in \cite{perkovic15_uai,perkovic16}. The only difference is the use of b-possibly causal paths as opposed to possibly causal paths. This similarity is intentional, as it makes our results easier to follow and further demonstrates that the results for $\DAG$s and $\CPDAG$s can be leveraged for \mpdag{}s. The proofs do not follow from previous results and require special consideration, especially due to the partially directed cycles that can occur in \mpdag{}s. In $\DAG$s and $\CPDAG$s our b-adjustment criterion reduces to the adjustment criteria from \cite{shpitser2012validity,shpitser2012avalidity} and \cite{perkovic15_uai,perkovic16}. For consistency however, we will refer to the b-adjustment criterion for all graph types.

To define our b-adjustment criterion, we first introduce the concept of the b-forbidden set for \mpdag{}s. The b-forbidden set contains all nodes that cannot be used for adjustment.
\begin{definition}$(\bfb{\g})$ Let $\mathbf{X}$ and $\mathbf{Y}$ be disjoint node sets in a \mpdag{} $\g$. We define the b-forbidden set relative to $(\mathbf{X,Y})$ as:
\begin{align}
   \bfb{\g} = \{ & W' \in \vars: W' \in \bPossDe(W,\g), \notag\\
   &\text{for some } W \notin \mathbf{X} \, \text{which lies on a} \notag \\
   &  \text{ proper b-possibly causal path } \notag\\
   & \text{from } \mathbf{X} \,\text{to}\, \mathbf{Y}\text{in } \g\}.\notag
\end{align} \label{def:forbidden nodes}
\end{definition}
\vspace{-.4cm}
\begin{definition}{(\textbf{b-adjustment criterion})}
   Let $\mathbf{X,Y}$ and $\mathbf{Z}$ be pairwise disjoint node sets in a \mpdag{} $\g$. Then $\mathbf{Z}$ satisfies the b-adjustment criterion relative to $(\mathbf{X,Y})$ in~$\g$ if:\vspace{-0.4cm}
   \begin{enumerate}[label = (\cctext*), leftmargin=0.5cm,align=left]
   \item\label{cond0} all proper b-possibly causal paths from $\mathbf{X}$ to $\mathbf{Y}$ start with a directed edge out of $\mathbf{X}$ in $\g$,
   \item\label{cond1} $\mathbf{Z} \cap \bfb{\g} = \emptyset$,
   \item\label{cond2} all proper b-non-causal definite status paths from $\mathbf{X}$ to $\mathbf{Y}$ are blocked by $\mathbf{Z}$~in~$\g$.
   \end{enumerate}  \label{def:gac-pdag}
\end{definition}
In Theorem~\ref{theorem:gac-pdag} we show that the b-adjustment criterion is sound and complete for adjustment.
Furthermore, in Theorem~\ref{theorem:constructive-set-mpdag} we show that there exists an adjustment set if and only if the specific set $\badjustb{\g}$ (Definition~\ref{def:constr-set-mpdag}) is an adjustment set relative to $(\mathbf{X,Y})$ in $\g$. 
\begin{theorem}
     Let $\mathbf{X,Y}$ and $\mathbf{Z}$ be pairwise disjoint node sets in a causal \mpdag{} $\g$. Then $\mathbf{Z}$ is an adjustment set (Definition~\ref{defadjustmentmpdag}) relative to $(\mathbf{X,Y})$ in $\g$ if and only if $\mathbf{Z}$ satisfies the b-adjustment criterion (Definition~\ref{def:gac-pdag}) relative to $(\mathbf{X,Y})$ in $\g$.
   \label{theorem:gac-pdag}
\end{theorem}
\begin{definition}$(\badjustb{\g})$
Let $\mathbf{X}$ and $\mathbf{Y}$ be disjoint node sets in a \mpdag{} $\g$. We define
\begin{align}
&\badjustb{\g}=  \nonumber \\
&\bPossAn(\mathbf{X} \cup \mathbf{Y}, \g)\setminus(\mathbf{X} \cup \mathbf{Y} \cup \bfb{\g}). \nonumber
\end{align}
\label{def:constr-set-mpdag}
\end{definition}
\vspace{-.4cm}
\begin{theorem}
 Let $\mathbf{X}$ and $\mathbf{Y}$ be disjoint node sets in a \mpdag{} $\g$. There exists a set that satisfies the b-adjustment criterion relative to $(\mathbf{X,Y})$ in $\g$ if and only if $\badjustb{\g}$ satisfies the b-adjustment criterion relative to $(\mathbf{X,Y})$ in $\g$.
\label{theorem:constructive-set-mpdag}
\end{theorem}
\begin{figure}[!tbp]
   \centering
   \begin{subfigure}{.16\textwidth}
     \centering
     \begin{tikzpicture}[>=stealth',shorten >=1pt,auto,node distance=2cm,main node/.style={minimum size=0.6cm,font=\sffamily\Large\bfseries},scale=0.6,transform shape]
     \node[main node]         (X)                        {$X$};
   \node[main node]         (V1) [left of= X]  		{$V_{1}$};
   \node[main node]         (V2) [below left of = X] 	{$V_{2}$};
   \node[main node]       	 (Y)  [below right of= V2] 	{$Y$};
   \draw[-] (V2) edge    (X);
   \draw[-] (V1) edge    (X);
   \draw[-] (X) edge    (Y);
   \draw[-] (V2) edge    (Y);
   \end{tikzpicture}
     \caption{}
     \label{visible:adj1}
   \end{subfigure}
   \unskip
   \vrule
   \begin{subfigure}{.16\textwidth}
     \centering
     \begin{tikzpicture}[>=stealth',shorten >=1pt,auto,node distance=2cm,main node/.style={minimum size=0.6cm,font=\sffamily\Large\bfseries},scale=.6,transform shape]
   \node[main node]         (X)                        {$X$};
   \node[main node]         (V1) [left of= X]  		{$V_{1}$};
   \node[main node]         (V2) [below left of = X] 	{$V_{2}$};
   \node[main node]       	 (Y)  [below right of= V2] 	{$Y$};
   \draw[<-] (V2) edge    (X);
   \draw[->] (V1) edge    (X);
   \draw[->] (X) edge    (Y);
   \draw[-] (V2) edge    (Y);
   \end{tikzpicture}
     \caption{}
     \label{visible:adj2}
   \end{subfigure}
   \unskip
   \vrule
   \begin{subfigure}{.16\textwidth}
     \centering
     \begin{tikzpicture}[>=stealth',shorten >=1pt,auto,node distance=2cm,main node/.style={minimum size=0.6cm,font=\sffamily\Large\bfseries},scale=.6,transform shape]
   \node[main node]         (X)                        {$X$};
   \node[main node]         (V1) [left of= X]  		{$V_{1}$};
   \node[main node]         (V2) [below left of = X] 	{$V_{2}$};
   \node[main node]       	 (Y)  [below right of= V2] 	{$Y$};
   \draw[-] (V2) edge    (X);
   \draw[<-] (V1) edge    (X);
   \draw[->] (Y) edge    (X);
   \draw[-] (V2) edge    (Y);
   \end{tikzpicture}
     \caption{}
     \label{visible:adj3}
   \end{subfigure}
   \caption{(a) $\CPDAG$ $\g[C]$, (b) \mpdag{} $\g_{1}$, (c) \mpdag{} $\g_{2}$ used in Example \ref{ex2}.}
   \label{figex2}
\end{figure}
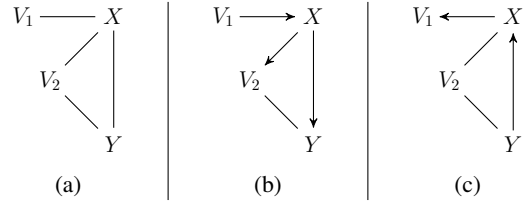
\begin{example}
We use this example to illustrate the b-adjustment criterion.
   Consider Figure~\ref{figex2} with a $\CPDAG$ $\g[C]$ in (a), and two \mpdag{}s $\g_{1}$ and $\g_{2}$ in $[\g]$ in (b) and (c). \Mpdag{} $\g_{1}$ ($\g_{2}$) can be obtained from $\g[C]$ by adding $V_1 \rightarrow X$ ($Y \rightarrow X$) as background knowledge and completing the orientation rules from Figure~\ref{fig:orientationRules}.

 $\g[C]$ is not b-amenable relative to $(X,Y)$ due to $X - Y$ and $X -V_2 -Y$. Hence, there is no adjustment set (no set satisfies Definition~\ref{def:gac-pdag}) relative to $(X,Y)$ in $\g$.

  On the other hand, $\g_{1}$ is b-amenable relative to $(X,Y)$ and $\bbf{\g_{1}} = \{ V_2, Y\}$. Since there are no b-non-causal paths from $X$ to $Y$ in $\g_{1}$ any set of nodes disjoint with $\bbf{\g_{1}} \cup \{X,Y\}$ satisfies the b-adjustment criterion relative to $(X,Y)$. Hence, all valid adjustment sets relative to $(X,Y)$ in $\g_{1}$ are $\emptyset$ and $\{V_1\}$.

 \Mpdag{} $\g_2$ is also b-amenable relative to $(X,Y)$ (since $X-V_2 -Y$ is a b-non-causal path). Since $X \leftarrow Y$ is in $\g_{2}$, $\bbf{\g_{2}} = \emptyset$. However, since $X \leftarrow Y$ is a proper b-non-causal definite status path from $X$ to $Y$ that cannot be blocked by any set of nodes, there is no adjustment set (no set satisfies Definition~\ref{def:gac-pdag}) relative to $(X,Y)$ in $\g_2$. Nevertheless, since $Y \notin \bPossDe(X,\g_{2})$, we can conclude that the total causal effect of $X$ on $Y$ in $\g_{2}$ is zero.
   \label{ex2}
\end{example}

\subsubsection{Constructing adjustment sets}

Checking whether $\mathbf{Z}$ is an adjustment set relative to $(\mathbf{X,Y})$ in $\g$ requires checking the three conditions in Definition~\ref{def:gac-pdag}: b-amenability, b-forbidden set and b-blocking. Checking \bamen{} or \bforb{} is computationally straightforward and depends only on constructing the set of all b-possible descendants of $\mathbf{X}$. 
Naively checking \bblck{}, however, requires keeping track of all paths between $\mathbf{X}$ and $\mathbf{Y}$. This scales very poorly with the size of the graph. To deal with this issue we rely on Lemma~\ref{lemma:eqb-pdag} in the supplement, which is analogous to Lemma~10 in \citealp{perkovic16}.

If $\g$ is b-amenable and $\mathbf{Z}$ satisfies \bforb{} relative to $(\mathbf{X,Y})$, then in order to verify that $\mathbf{Z}$ satisfies \bblck{} in $\g$ it is enough to verify that $\mathbf{Z}$ satisfies \bblck{} in one $\DAG$ in $[\g]$ (\ref{l:eqb3-pdag} in Lemma~\ref{lemma:eqb-pdag} in the supplement).
In \cite{vanconstructing}, the authors propose fast algorithms to verify \bblck{} and construct adjustment sets in $\DAG$s. Using the above mentioned result we can leverage these for use in \mpdag{}s.

The computational complexity of finding one $\DAG$ $\g[D]$ in $[\g]$ is $O(|\mathbf{V}||\mathbf{E}|)$, as shown in \cite{dorTarsi92}. The computational complexity of verifying whether a set satisfies the b-adjustment criterion in $\g[D]$  is $O(|\mathbf{V}| +|\mathbf{E}|)$, as shown in \cite{vanconstructing}. Furthermore, the computational complexity of listing all or all minimal adjustment sets in $\g[D]$, is at most polynomial in $|\mathbf{V}|$ and $|\mathbf{E}|$ per output set, as shown in \cite{vanconstructing}. The complexity of verifying \bblck{} in $\g$ when exploiting the results of Lemma~\ref{lemma:eqb-pdag}, \cite{dorTarsi92} and \cite{vanconstructing} is polynomial in $|\mathbf{V}|$ and $|\mathbf{E}|$ and listing all or all minimal adjustment sets in $\g$ is polynomial in $|\mathbf{V}|$ and $|\mathbf{E}|$ per output set.

\subsection{IDA AND JOINT-IDA IN \MPDAG{}s} \label{sec:ida}

It is not always possible to find an adjustment set relative to $(\mathbf{X,Y})$ in a \mpdag{} $\g$. For example, if the total causal effect of $\mathbf{X}$ on $\mathbf{Y}$ differs in some distinct $\DAG$s in $[\g]$, then this effect is not identifiable in $\g$, and is certainly not identifiable via adjustment.

The IDA algorithm for $\CPDAG$s from \cite{MaathuisKalischBuehlmann09} was developed with precisely this issue in mind.
In order to estimate the possible total causal effects of a node $X$ on a node $Y$ based on a $\CPDAG$ $\g[C]$ one can consider listing all $\DAG$s in $[\g[C]]$ and estimating the total causal effect of $X$ on $Y$ in each. Since this effect may differ between different $\DAG$s in $\g$, the output of such an algorithm is a multi-set of possible total causal effects of $X$ on $Y$.
The joint-IDA algorithm for $\CPDAG$s from \cite{nandy2014estimating} employs the same idea to estimate the possible total joint causal effect of a node set $\mathbf{X}$ on a node $Y$.

The IDA and joint-IDA algorithms use all possible (joint) parent sets of $\mathbf{X}$ in the $\DAG$s in $[\g]$. If $Y$ is not a parent of $X$ in a causal $\DAG$ $\g[D]$, it is well known that $\Pa(X,\g[D])$ is an adjustment set relative to $(X,Y)$ in $\g[D]$ \citep{Pearl2009}. 

The (joint-)IDA algorithm has different algorithmic variants for finding (joint) parent sets: global, local and semi-local. The global variant applies the above mentioned idea of listing all $\DAG$s represented by a $\CPDAG$ $\g[C]$ in order to find all possible (joint) parent sets of $\mathbf{X}$. This method can be applied directly to a \mpdag{} $\g$, but scales poorly with the number of undirected edges in $\g$.
The local variant of IDA from \cite{MaathuisKalischBuehlmann09} and the semi-local variant of joint-IDA from \cite{nandy2014estimating} dramatically reduce the computational complexity of the algorithm. However, they are not applicable to \mpdag{}s, as we demonstrate in Examples~\ref{ex:local-ida} and~\ref{ex:semiloc-jida}.

\begin{algorithm}[!t]
\TitleOfAlgo{Semi-locally find all joint parent sets of $\mathbf{X} = (X_1,\dots ,X_k), k \ge 1,$ in a \mpdag{} $\g$.}
 \KwData{\mpdag{} $\g$, $\mathbf{X} = \{X_1,\dots ,X_k\}, k \ge 1$}
 \KwResult{Multi-set $\mathbf{PossPa}$ of all joint parent sets of $\mathbf{X}$}
	$\mathbf{PossPa}=\emptyset$; \\
	$\mathbf{Sib_1} =\{ A: A - X_1 \text{ in } \g \}$; \\
	\If{$k >1$}{
 $\mathbf{Sib}_i = \{ A: A - X_i \text{ in } \g \} \setminus \{X_1, \dots, X_{i-1} \}$; \\
 }
 \ForAll{$\mathbf{S}_i \subseteq \mathbf{Sib}_i$, $i = 1 , \dots , k$,}{
$ \begin{aligned}
  \mathbf{LocalBg} = &\cup_{i=1}^k \{ A \rightarrow X_i : A \in \mathbf{S}_i\} \cup \\
  &\cup_{i=1}^k \{ X_i \rightarrow A : A \in \mathbf{Sib}_i \setminus \mathbf{S}_i\};
 \end{aligned}$\\
 \If{$\ConstructMaxPDAG(\g,\mathbf{LocalBg}) \neq \text{FAIL}$}{
  add $(\Pa(X_1,\g) \cup \mathbf{S}_1 ,\dots , \Pa(X_k,\g) \cup \mathbf{S}_k)$ to $\mathbf{PossPa}$;\\
  }
}
\label{algo:ida}
\vspace{-.3cm}
\end{algorithm}

In Algorithm~\ref{algo:ida}, we present our semi-local variant for finding possible (joint) parent sets in \mpdag{}s. This algorithm exploits the soundness and completeness of orientations in the \mpdag{} \citep{meek1995causal}. To find a possible parent set of $X$, Algorithm~\ref{algo:ida} simply considers all $\mathbf{S} \subseteq \Sib(X,\g)$ and imposes $S \rightarrow X$ and $X \rightarrow \bar{S}$, for all $S \in \mathbf{S}$, and $\bar{S} \in \Sib(X,\g)\setminus \mathbf{S}$ as background knowledge called $\mathbf{LocalBg}$. If this background knowledge is consistent with $\g$, then $\mathbf{S} \cup \Pa(X,\g)$ is a possible parent set of $X$. In case of joint interventions $(|\mathbf{X}|=k>1)$, Algorithm~\ref{algo:ida} does this for every $X_i \in \mathbf{X}, i \in \{1,\dots,k\}$ and $(\Pa(X_1,\g) \cup \mathbf{S}_1 ,\dots , \Pa(X_k,\g) \cup \mathbf{S}_k)$ is the possible joint parent set of $\mathbf{X}$.

Since the local IDA and semi-local joint-IDA are not applicable to \mpdag{}s, we can only compare the computational complexity of our Algorithm~\ref{algo:ida} with the local IDA and semi-local joint-IDA by applying them all to a $\CPDAG$. In this case, Algorithm~\ref{algo:ida} will in general be slower than local IDA, as it requires closing the orientation rules of \cite{meek1995causal}, but also in general somewhat faster than semi-local joint IDA, since it does not require orienting an entire undirected component of the $\CPDAG$. We compare the runtimes of local IDA and our semi-local IDA in an empirical study in Section~\ref{sec:study} of the supplement.

\subsubsection{Examples}

\begin{example}
Consider the \mpdag{} $\g$ in Figure~\ref{mpdag11}. Suppose we want to estimate the total causal effect of $C$ on $A$ in $\g$.
All possible parent sets of $C$ according to Algorithm~\ref{algo:ida} are: $\emptyset , \{D\}, \{B,D\}$. These are also all possible parent sets of $C$ according to Figure~\ref{alldagmpdag}.

We now show that local IDA cannot be applied.
For every $\mathbf{S} \subseteq \Sib(C,\g)$, local IDA from \cite{MaathuisKalischBuehlmann09} orients $S -C$ as $S \rightarrow C$ for every $S \in \mathbf{S}$. If this does not introduce a new unshielded collider $\rightarrow C \leftarrow$ it returns $\Pa(C,\g) \cup \mathbf{S}$ as a valid parent set.  Local IDA returns the following sets as the possible parent sets of $C$: $\emptyset , \{B\}, \{D\}, \{B,D\}$. However, parent set $\{B\}$ means adding $B \rightarrow C$ and $C \rightarrow D$ and this introduces a cycle due to $D \rightarrow B$. Hence, $\{B\}$ will never be a parent set of $C$ in $[\g[C]]$ and local IDA is not valid for \mpdag{}s.
\label{ex:local-ida}
\end{example}

\begin{example}
Consider again the \mpdag{} $\g$ in Figure~\ref{mpdag11}. Suppose we want to estimate the total causal effect of $(C,D)$ on $A$ in $\g$.
All possible joint parent sets of $(C,D)$ according to Algorithm~\ref{algo:ida} are: $(\emptyset,\{C\}), (\{D\},\emptyset), (\{B,D\},\emptyset), (\{B,D\},\{A\})$. These are also all possible joint parent sets of $(C,D)$ according to Figure~\ref{alldagmpdag}.

Semi-local joint-IDA from \cite{nandy2014estimating} would attempt to learn all possible joint parent sets of $(C,D)$ by orienting the undirected component $B-C-D-A$ on $\g$ into all possible $\DAG$s without unshielded colliders. However, it is not possible to orient $B-C-D-A$ into a $\DAG$ without creating an unshielded collider. Hence, semi-local joint-IDA is not valid for \mpdag{}s.
\label{ex:semiloc-jida}
\end{example}

\section{IMPLEMENTATION AND SIMULATION STUDY} \label{sec:sim}

We investigated the effect of adding background
knowledge to a $\CPDAG$ in a simulation study using \texttt{R (3.3.3)} and the R-package \texttt{pcalg (2.4-6)} \citep{kalischpcalg}. The following functions were added or modified: \texttt{isValidGraph()}, \texttt{addBgKnowledge()}, \texttt{adjustment()}, \texttt{gac()}, \texttt{ida()}, \texttt{jointIda()}. Details about the simulation can be found in Appendix A of the \texttt{pcalg} package vignette: ``An Overview of the \texttt{pcalg} Package for \texttt{R}'', available on CRAN.
 
We first sampled $1000$ settings that were used
to generate graphs later on. The following settings were drawn uniformly at random:
the number of nodes $p \in \{20, 30, ..., 100\}$, and the expected neighborhood size $E[N] \in \{3, 4, ...,10\}$.

For each of these $1000$ settings, $20$ $\DAG$s were randomly generated and then transformed into the
corresponding $\CPDAG$. This resulted in
$20\ 000$ $\CPDAG$s. For each graph, we randomly chose a node $X$ and then
randomly chose a node $Y$ that was connected to $X$ but was not a parent of
$X$ in the true underlying $\DAG$. Furthermore, for each DAG we generated a data set with sample size $n=200$.

For each of the $20\ 000$ $\CPDAG$s, we then generated additional \mpdag{}s by
replacing a fraction of randomly chosen undirected edges by the
true directed edges of the underlying $\DAG$ and applying the orientation
rules in Figure~\ref{fig:orientationRules} to further orient edges if possible. The fraction of background knowledge varied through all
values in the set $\{0, 0.1, \dots, 0.9, 1\}$, resulting in eleven \mpdag{}s. Note that a fraction of $0$ corresponds to the $\CPDAG$ and a fraction of $1$ corresponds
to the true underlying $\DAG$.

For each of the $220\ 000$ \mpdag{}s $\g$ we analyzed two questions:
\begin{description}
\item[Q1:] Is there a set that satisfies the b-adjustment criterion relative to
  $(X,Y)$ in $\g$?
\item[Q2:] What is the multi-set of possible total causal effects of $X$ on $Y$ given
  $\g$ and the sampled data of the corresponding $\DAG$? In particular,
  what is the number of unique estimates in the multi-set?
\end{description}

\begin{figure}[!tb]
\centering
\includegraphics[scale=.4]{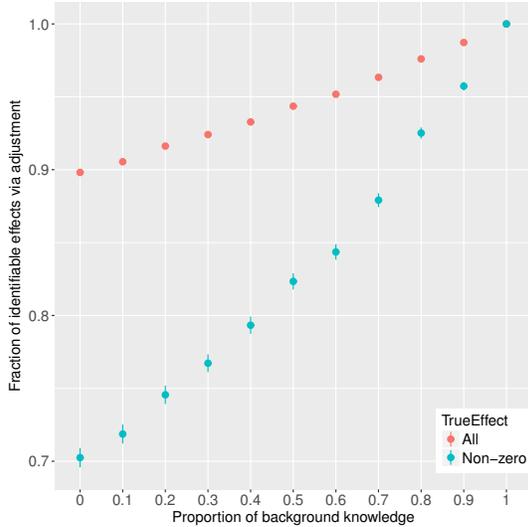}
\caption{Fraction of total causal effects that are identifiable with our adjustment criterion, with respect to the proportion of
background knowledge.
Red: all simulations;
blue: those simulations where the
true total causal effect is non-zero. Error bars indicate standard errors and are
partly too small to be seen.}
\label{fig:Q1}
\end{figure}

The results for question (\textbf{Q1}) are shown in Figure~\ref{fig:Q1}.
We see that without any background knowledge around $90\%$ of all total causal effects
could be identified via adjustment, while around $70\%$ of the non-zero total causal effects could be
identified via adjustment. When the proportion of background knowledge increases, the
fraction of total causal effects that we identify via adjustment increases, both for all total causal effects and for the
non-zero total causal effects. With $100\%$ background knowledge, the \mpdag{}
is identical to the true underlying $\DAG$ and identification of the total causal effect of $X$ on $Y$ via covariate adjustment is always possible, since $X$ is a single intervention and $Y$ is not a parent of $X$ in the $\DAG$ \citep{Pearl2009}.

\begin{figure}[!t]
\centering
\includegraphics[scale=.4]{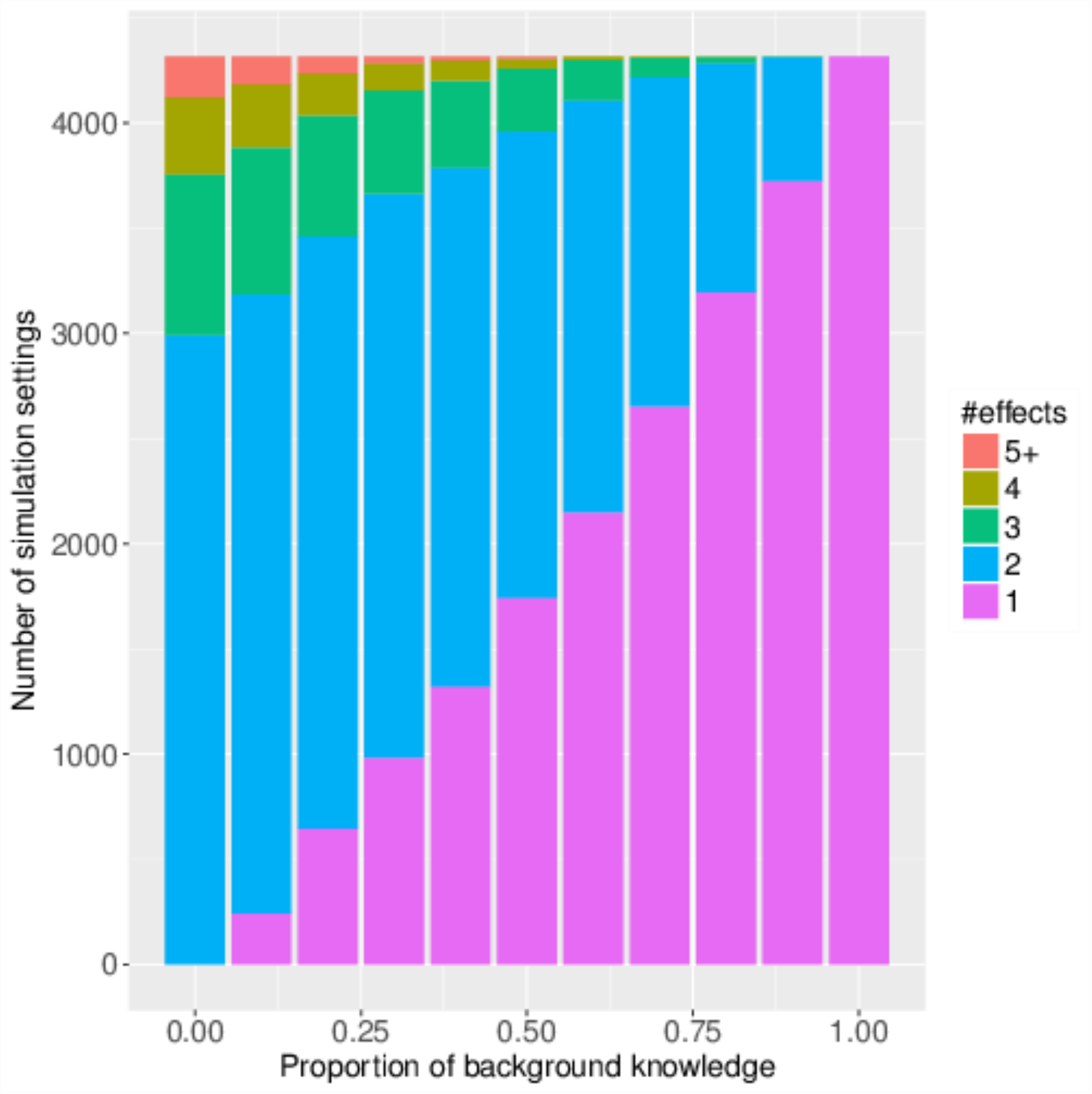}
\caption{Stacked barplot of the number of unique possible total causal effects in the multi-sets of our semi-local
  IDA, with respect to the proportion of background knowledge. The results are restricted to the $4315$ $\CPDAG$s (out of $20\ 000$) that have more than one unique estimated possible total causal effect in the multi-set.}
\label{fig:Q2}
\vspace{-.4cm}
\end{figure}

The results for questions (\textbf{Q2}) are shown in Figure~\ref{fig:Q2},
restricting ourselves to the $4315$ $\CPDAG$s that had more than one
unique estimate in the multi-set.
When including
background knowledge, the fraction of multi-sets with exactly one unique
estimate gradually increases. Finally, with $100\%$ background knowledge,
the \mpdag{} is identical to the true underlying $\DAG$ and the multi-set
always contains one unique element.

\section{DISCUSSION} \label{sec:disc}

Although \mpdag{}s typically contain more
orientation information than $\CPDAG$s, this additional information has not been fully exploited in practice, due to a lack of understanding and methodology for \mpdag{}s. Our paper aims to make an important step in bridging this gap and in opening the way for the use of \mpdag{}s in practice.

This paper introduces various tools for working with \mpdag{}s. In particular, we are now able to read off possible ancestral relationships directly from a \mpdag{} and to estimate (possible) total causal effects when a \mpdag{} is given. Since $\CPDAG$s and $\DAG$s are special cases of \mpdag{}s, our b-adjustment criterion and semi-local (joint-)IDA methods for \mpdag{}s generalize existing results for $\CPDAG$s and $\DAG$s \citep{MaathuisKalischBuehlmann09,shpitser2012validity,perkovic15_uai,perkovic16,
nandy2014estimating}. All methods are implemented in the \texttt{R} package \texttt{pcalg}.

The examples and the simulation study in our paper involve \mpdag{}s generated by adding background knowledge to a $\CPDAG$. Nevertheless, we emphasize that \mpdag{}s can arise in many different ways, e.g., by
adding background knowledge before structure learning \citep{tetrad1998}, by structure learning from a combination of observational and interventional data \citep{hauserBuehlmann12,wang2017permutation}, or by structure learning for certain restricted model classes \citep{hoyer08,ernestroth2016,eigenmann17}.

It would be interesting to extend our methods to settings with hidden variables, e.g., considering partial ancestral graphs ($\PAG$s;  \citealp{richardson2002ancestral,ali2012towards}) with background knowledge. An important missing link for such an extension is a clear understanding of $\PAG$s with background knowledge. In particular, one would need to develop complete orientation rules for $\PAG$s with background knowledge, analogous to the work of \cite{meek1995causal}. Once this is in place, it seems feasible to generalize graphical criteria for covariate adjustment in $\PAG$s \citep{perkovic15_uai,perkovic16} and an IDA type method for $\PAG$s, called LV-IDA \citep{malinsky2017estimating}.

\subsubsection*{Acknowledgements}
This work was supported in part by Swiss NSF Grant 200021\_172603.

\newpage
\appendix

\section*{SUPPLEMENT}

This is the supplement of the paper ``Interpreting and using CPDAGs with background knowledge", which we refer to as the ``main text".

\section{PRELIMINARIES} \label{sec1}

\textbf{Paths.} If $p = \langle X_1, X_2, \dots , X_k, \rangle, k \ge 2$ is a path, then with $-p$ we denote the path $\langle X_k, \dots , X_2, X_1 \rangle$. The \textit{length} of a path equals the number of edges on the path. We denote the concatenation of paths by $\oplus$, so that for example $p = p(X_1,X_{m}) \oplus p(X_{m},X_{k})$ for $1 \le m \le k$.

\begin{definition} (\textbf{Distance-from-$\mathbf{Z}$})
Let $\mathbf{X,Y}$ and $\mathbf{Z}$ be pairwise disjoint node sets in a \mpdag{} $\g$. Let $p$ be a path from $\mathbf{X}$ to $\mathbf{Y}$ in~$\g$ such that every collider $C$ on $p$ has a b-possibly causal path to $\mathbf{Z}$. Define the $\distancefrom{\mathbf{Z}}$ of collider $C$ to be the length of a shortest b-possibly causal path from $C$ to $\mathbf{Z}$, and define the $\distancefrom{\mathbf{Z}}$ of $p$ to be the sum of the distances from $\mathbf{Z}$ of the colliders on $p$.
\label{def:distance from Z}
\end{definition}

\begin{lemma} \citep[Lemma~A.7 in][]{ernestroth2016}
Let $X$ and $Y$ be nodes in a  \mpdag{} $\g$ such that $X - Y$ is in $\g$. Let $\g' = \ConstructMaxPDAG(\g,\{X \rightarrow Y\})$. For any $Z,W \in \mathbf{V}$ if $Z \rightarrow W$ is in $\g'$ and $Z -W$ is in $\g$, then $W \in \De(Y,\g')$.
\label{lemma:alwaydesc}
\end{lemma}

\begin{lemma} \citep[cf.\ Lemma~A.8 in][]{ernestroth2016}
Let $X$ be a node in a \mpdag{} $\g$. Then there is a \mpdag{} $\g'$ in $[\g]$ such that $X \rightarrow S$ is in $\g'$ for all $S \in \Sib(X,\g)$. 
\label{lemma:no-new-into}
\end{lemma}

\section{PROOFS FOR SECTION~\ref{sec:understanding}} \label{sec:proofunderstanding}

\begin{proofof}[Lemma~\ref{lemma:posdir-path}]
Since $\pstar = \langle X=V_0, \dots , V_k=Y \rangle$, $k \geq 1$ is b-non-causal in $\g$, we have $V_i \leftarrow V_j$ in $\g$ for some $i,j$ such that $ 0 \le i < j \le k$. Let $\g[D]$ be an arbitrary $\DAG$ in $[\g]$ and let $p$ be the path corresponding to $\pstar$ in $\g[D]$.
Since $V_i \leftarrow V_j$ in $\g[D]$, $p(V_i, V_j)$ is non-causal from $V_i$ to $V_j$ in $\g[D]$.  Hence, $p$ is b-non-causal in $\g[D]$.
\end{proofof}
\begin{proofof}[Lemma~\ref{lemma:def-stat-posdir}]
One direction is trivial and we only prove that if there is no $V_i \leftarrow V_{i+1}$, for $i \in \{1,\dots, k-1\}$  in $\g$, then $\pstar$ is b-possibly causal in $\g$. Suppose for a contradiction that $\pstar$ is b-non-causal, that is, there is an edge $V_{j} \leftarrow V_{r}$, for $1 \le j < r \le k$, where $r \neq j +1$. 

Since there is no $V_i \leftarrow V_{i+1}$ for any $i \in \{1,\dots, k-1\}$ in $\g$, $V_i - V_{i+1}$ or $V_i \rightarrow V_{i+1}$ is in $\g$ for every $i \in \{1,\dots, k-1\}$.
Let $\g[D]$ be a $\DAG$ in $[\g]$ that contains $V_1 \rightarrow V_2$ and let $p$ be the path corresponding to $\pstar$ in $\g[D]$. 
Since $\pstar$ is of definite status in $\g$ and since no $V_i \leftarrow V_{i+1}$, $i \in \{1,\dots, k-1\}$ is in $\g$, it follows that $\pstar$ contains only definite non-colliders. Then since $V_1 \rightarrow V_2$ is on $p$, $p$ is a causal path in $\g[D]$.
But then $p(V_{j},V_{r})$ together with $V_j \leftarrow V_r$ create a directed cycle in $\g[D]$. 
\end{proofof}
Lemma~\ref{lemma:unshielded-analog} is analogous to Lemma~B.1 in \cite{zhang2008completeness} and the proof follows the same reasoning as well.
\begin{proofof}[Lemma~\ref{lemma:unshielded-analog}] The proof is by induction on the length of $p$. Let $p=\langle X= V_1, \dots , V_k =Y \rangle$.
Suppose that $k=3$. Then either $p$ is unshielded, or there is an edge  $X - Y$ or $X \rightarrow Y$ in $\g$ ($X \leftarrow Y$ is not in $\g$ since $p$ is b-possibly causal). 

For the induction step suppose that the lemma holds for paths of length $n-1$ and let $k=n$. Then either $p$ is unshielded, or there is a node $V_i$, $i >1$ on $p$, such that $V_{i-1} - V_{i+1}$ or $V_{i-1} \rightarrow V_{i+1}$ is in $\g$ ($V_{i-1} \leftarrow V_{i+1}$ is not in $\g$ since $p$ is b-possibly causal ). Then $p' = p(X,V_{i-1}) \oplus \langle V_{i-1}, V_{i+1} \rangle \oplus p(V_{i+1},Y)$ is a b-possibly causal path from $X$ to $Y$ of length $n-1$ and $p'$ is a subsequence of $p$.  
\end{proofof}
The following lemma is analogous to Lemma 7.2 in \cite{maathuis2013generalized} and follows directly from our definitions of b-possibly causal paths and definite status paths.
\begin{lemma}
 Let $p =\langle V_1, \dots , V_k \rangle$ be a b-possibly causal definite status path in a \mpdag{} $\g$. If there is a node $i \in \{1, \dots, n-1\}$ such that $V_i \rightarrow V_{i+1}$, then $p(V_i, V_k)$ is a causal path in $\g$.
\label{lemma:unshielded-edgeanalog}
\end{lemma}

\section{PROOFS FOR SECTION \ref{sec:adjust-mpdag} OF THE MAIN TEXT} \label{sec:proofadjust-mpdag}

\subsection{PROOF OF THEOREM~\ref{theorem:gac-pdag}}

\begin{figure}[!tbp]
    \centering
      \begin{tikzpicture}[>=stealth',shorten >=1pt,node distance=3cm, main node/.style={minimum size=0.4cm}]
     \node[main node]         (T34) {\textbf{Theorem \ref{theorem:gac-pdag}}};
    \node[main node,yshift=1cm] (L52) at (T34) {Lemma \ref{lemma:amen-pdag}};
    \node[main node]         (L53)  [left of= T34]          {Lemma \ref{lemma:eqpdag}};
    \node[main node,yshift=-1cm] (L54) at (T34) {Lemma~\ref{lemma:eqb-pdag}};
    \node[main node]            (L56)  [left of= L54]   {Lemma \ref{lemma:2-pdag}};
    \node[main node]            (L57)  [left of= L56,text width=2.2cm,align=center]   {Lemma \ref{lemma:1-pdag}};
    \node[main node,yshift=1cm]            (L58) at (L57)   {Lemma \ref{lemma:non-causal-path-pdag}};
    \draw[->] (L52) edge    (T34);
    \draw[->] (L53) edge    (T34);
    \draw[->] (L54) edge    (T34);
    \draw[->] (L56) edge    (L54);
    \draw[->] (L57) edge    (L56);
    \draw[->] (L58) edge    (L57);
    \end{tikzpicture}
    \caption{Proof structure of Theorem \ref{theorem:gac-pdag}.}
    \label{figproof}
\end{figure}
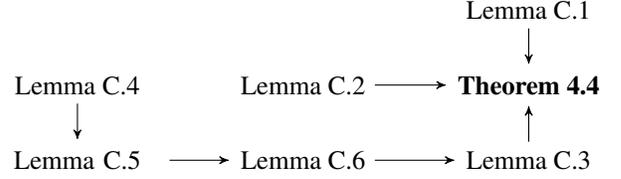

Figure~\ref{figproof} shows how all lemmas fit together to prove Theorem \ref{theorem:gac-pdag}. Theorem~\ref{theorem:gac-pdag} is closely related to Theorem 5 for $\CPDAG$s from \cite{perkovic16}. Since every $\CPDAG$ is a \mpdag{}, all the results presented here subsume the existing  results for $\CPDAG$s. Throughout, we inform the reader when our results and proofs differ from the existing ones for $\CPDAG$s.

\begin{proofof}[Theorem~\ref{theorem:gac-pdag}]
This proof is basically the same as the proof of Theorem 5 from \cite{perkovic16}, except that instead of using Lemmas 8, 9 and 10 from \cite{perkovic16}, we need to use Lemmas ~\ref{lemma:amen-pdag}, \ref{lemma:eqpdag} and \ref{lemma:eqb-pdag}. We give the entire proof for completeness.

 Suppose first that $\mathbf{Z}$ satisfies the b-adjustment criterion relative to $(\mathbf{X,Y})$ in the \mpdag{} $\g$. We need to show that $\mathbf{Z}$ is an adjustment set (Definition~\ref{defadjustmentmpdag}) relative to $(\mathbf{X,Y})$ in every $\DAG$ $\g[D]$ in $[\g]$. By applying Lemmas~\ref{lemma:amen-pdag},~\ref{lemma:eqpdag} and~\ref{lemma:eqb-pdag} in turn, it directly follows that $\mathbf{Z}$ satisfies the b-adjustment criterion relative to $(\mathbf{X,Y})$ in any $\DAG$ $\g[D]$ in $[\g]$.
  Since the b-adjustment criterion reduces to the adjustment criterion \citep{shpitser2012validity,shpitser2012avalidity} in $\DAG$s and the adjustment criterion is sound for $\DAG$s, $\mathbf{Z}$ is an adjustment set relative to $(\mathbf{X,Y})$ in~$\g[D]$.

   To prove the other direction, suppose that $\mathbf{Z}$ does not satisfy the b-adjustment criterion relative to $(\mathbf{X,Y})$ in~$\g$. First, suppose that $\g$ violates \bamen{} relative to $(\mathbf{X},\mathbf{Y})$. Then by Lemma~\ref{lemma:amen-pdag}, there is no adjustment set relative to $(\mathbf{X},\mathbf{Y})$ in~$\g$.
   Otherwise, suppose $\g$ is b-amenable relative to $(\mathbf{X,Y})$. Then $\mathbf{Z}$ violates \bforb{} or \expandafter\ignorespaces\bblck{}. 
   We need to show $\mathbf{Z}$ is not an adjustment set in at least one $\DAG$ $\g[D]$ in $[\g]$.
   Suppose $\mathbf{Z}$ violates \expandafter\ignorespaces\forb. Then by Lemma~\ref{lemma:eqpdag}, it follows that there exists a $\DAG$ $\g[D]$ in $[\g]$ such that $\mathbf{Z}$ does not satisfy the b-adjustment criterion relative to $(\mathbf{X,Y})$ in~$\g[D]$.   Since the b-adjustment criterion reduces to the adjustment criterion \citep{shpitser2012validity,shpitser2012avalidity} in $\DAG$s and the adjustment criterion is complete for $\DAG$s, it follows that $\mathbf{Z}$ is not an adjustment set relative to $(\mathbf{X,Y})$ in~$\g[D]$.
   Otherwise, suppose $\mathbf{Z}$ satisfies \expandafter\ignorespaces\forb, but violates \expandafter\ignorespaces\bblck{}. Then by Lemma~\ref{lemma:eqb-pdag}, it follows that there is a $\DAG$ $\g[D]$ in $[\g]$ such that $\mathbf{Z}$ does not satisfy the b-adjustment criterion relative to $(\mathbf{X,Y})$ in~$\g[D]$.   Since the b-adjustment criterion reduces to the adjustment criterion \citep{shpitser2012validity,shpitser2012avalidity} in $\DAG$s and the adjustment criterion is complete for $\DAG$s, it follows that $\mathbf{Z}$ is not an adjustment set relative to $(\mathbf{X,Y})$ in~$\g[D]$.
\end{proofof}

\begin{lemma}
Let $\mathbf{X}$ and $\mathbf{Y}$ be disjoint node sets in a \mpdag{} $\g$. If $\g$ violates \bamen{} relative to $(\mathbf{X},\mathbf{Y})$, then there is no adjustment set relative to $(\mathbf{X},\mathbf{Y})$ in~$\g$. 
\label{lemma:amen-pdag}
\end{lemma}

\begin{proof} 
This lemma is related to Lemma~8 from \cite{perkovic16}. The proofs are not the same due to the differences between $\CPDAG$s and \mpdag{}s. We will point out where the two proofs diverge.

   Suppose that $\g$ violates \bamen{} relative to $(\mathbf{X},\mathbf{Y})$. We will show that in this case one can find $\DAG$s $\g[D]_1$ and $\g[D]_2$ in $[\g]$, such that there is no set that satisfies the b-adjustment criterion relative to $(\mathbf{X,Y})$ in both $\g[D]_1$ and $\g[D]_2$.
 
 Since $\g$ is not b-amenable relative to $(\mathbf{X,Y})$, there is a proper b-possibly causal path $\pstar[q]$ from a node $X \in \mathbf{X}$ to a node $Y \in \mathbf{Y}$ that starts with a undirected edge.
   Let $\pstar[q'] = \langle X= V_0,V_1,\dots,V_k=Y\rangle, k \ge 1$ (where $V_1 = Y$ is allowed) be a shortest subsequence of $\pstar[q]$ such that $\pstar[q']$ is also a proper b-possibly causal path that starts with a undirected edge in~$\g$. 
  
   Suppose first that $\pstar[q']$ is of definite status in $\g$. Let $\g[D]_1$ be a $\DAG$ in $[\g]$ that contains $X \rightarrow V_1$  and let $\g[D]_2$ be a $\DAG$ in $[\g]$ that has no additional edges into $V_1$ as compared to $\g$ (Lemma~\ref{lemma:no-new-into}). Then the path corresponding to $\pstar[q']$ in $\g[D]_1$ is causal, whereas the path corresponding to $\pstar[q']$ in $\g[D]_2$ is b-non-causal and contains no colliders. Hence, no set can satisfy both \bforb{} in $\g[D]_1$ and \bblck{} in $\g[D]_2$  relative to $(\mathbf{X,Y})$.
   
Otherwise, $\pstar[q']$ is not of definite status in $\g$. In the proof of Lemma 8 from \cite{perkovic16}, the authors show that if $\pstar[q']$ is not of definite status, this leads to a contradiction. However, $\pstar[q']$  can be of non-definite status in $\g$ and this is where the proofs diverge. 

By Lemma~\ref{lemma:unshielded-analog}, $\pstar[q'](V_1,Y)$ must be unshielded and hence, of definite status in $\g$, since otherwise we can choose a shorter b-possibly causal path.
Since $\pstar[q']$ is not of definite status and $\pstar[q'](V_1,Y)$ is of definite status, it follows that $V_1$ is not of definite status on $\pstar[q']$. Then $\langle X, V_1, V_2 \rangle$ is a shielded triple. By choice of $\pstar[q']$, $X - V_1$ is in $\g$. Additionally, since $V_1$ is not of definite status on $\pstar[q']$, $V_1 - V_2$ must be in $\g$.
This implies that $X - V_1 -V_2$ is in $\g$ and $X \in \Adj(V_2,\g)$. Moreover, we must have $X \to V_2$, since $X - V_2$ contradicts the choice of $\pstar[q']$, and $X \leftarrow V_2$ contradicts that $\pstar[q']$ is b-possibly causal in $\g$. 

Let $\g[D]_1$ be a $\DAG$ in $\g$ that has no additional edges into $V_1$ as compared to $\g$ (Lemma~\ref{lemma:no-new-into}). Let $q_1$ be the path corresponding to $\pstar[q']$ in $\g[D]_1$. Then $q_1$ is of the form $X \leftarrow V_1 \rightarrow V_2 \rightarrow \dots \rightarrow Y$ in $\g[D]_1$. Since $X \rightarrow V_2 \rightarrow \dots \rightarrow Y$ is a proper causal path in $\g[D]_1$, $\{V_2, \dots, V_{k-1}\} \subseteq \bfb{\g[D]_{1}}$. Hence, any set that satisfies \bblck{} and \bforb{} relative to $(\mathbf{X,Y})$ in $\g[D]_{1}$ must contain $V_1$ and not $\{V_2,\dots,V_{k-1}$. 

Let $\g[D]_2$ be a $\DAG$ in $\g$ that has no additional edges into $V_2$ as compared to $\g$ (Lemma~\ref{lemma:no-new-into}). Let $q_2$ be the path corresponding to $\pstar[q']$ in $\g[D]_2$. Since $X \rightarrow V_2 \rightarrow V_1$ is in $\g[D]_2$, $X \rightarrow V_1$ is in $\g[D]_2$ (Rule $R2$). Then $q_2$ is of the form $ X \rightarrow V_1 \leftarrow V_2 \rightarrow \dots \rightarrow Y$ in $\g[D]_2$. Hence, any set that satisfies \bforb{} and \bblck{} {} relative to $(\mathbf{X,Y})$ in $\g[D]_{1}$, violates \bblck{} {} relative to $(\mathbf{X,Y})$ in $\g[D]_{2}$.
\end{proof}

\begin{lemma}
  Let $\mathbf{X}$ and $\mathbf{Y}$ be disjoint node sets in a \mpdag{} $\g$. If $\g$ is b-amenable relative to $(\mathbf{X},\mathbf{Y})$, then the following statements are equivalent:
\begin{enumerate}[label = (\roman*)]
\item\label{l:eqa1-pdag} $\mathbf{Z}$ satisfies \bforb{} (see Definition~\ref{def:gac-pdag}) relative to $(\mathbf{X},\mathbf{Y})$ in~$\g$.
\item\label{l:eqa2-pdag}  $\mathbf{Z}$ satisfies \bforb{} relative to $(\mathbf{X},\mathbf{Y})$ in every $\DAG$ in $[\g]$.
\end{enumerate}
   \label{lemma:eqpdag}
\end{lemma}

\begin{proof}
This lemma is related to Lemma~10 from \cite{perkovic16}.

 By Lemma~\ref{lemma:posdir-path} $\bfb{\g[D]} \subseteq \bfb{\g}$ hence, \ref{l:eqa1-pdag}$\Rightarrow$\ref{l:eqa2-pdag} holds. We now prove $\neg$~\ref{l:eqa1-pdag}$\Rightarrow\neg$~\ref{l:eqa2-pdag}.

Let $V \in \mathbf{Z} \cap \bfb{\g}$. Then $V \in \bPossDe(W,\g)$ for some $W= V_i$ on a proper b-possibly causal path $p = \langle X=V_0, V_1,\dots, V_k=Y \rangle, 1 \le i \le k$ from $X \in \mathbf{X}$ to $Y \in \mathbf{Y}$. Let $q=p(X,W)$, $r=p(W,Y)$ and let $s$ be a b-possibly causal path from $W$ to $V$, where $r$ and $s$ are allowed to be of zero length (if $W=Y$ and/or $W=V$). 

   Let $q'$, $r'$ and $s'$ be subsequences of $q$, $r$ and $s$ that form unshielded b-possibly causal paths, with $r'$ and $s'$ possibly of zero length (Lemma~\ref{lemma:unshielded-analog}). Then $q'$ must start with a directed edge, otherwise $q' \oplus r'$ would violate \bamen{}.
 Hence, $q'$ must be causal in~$\g$ (Lemma~\ref{lemma:unshielded-edgeanalog}). 

   Let $\g[D]$ be a $\DAG$ in $[\g]$ that has no additional edges into $W$ as compared to $\g$ (Lemma~\ref{lemma:no-new-into}). Then since $r'$ and $s'$ are unshielded and b-possibly causal, the paths corresponding to $r'$ and $s'$ in $\g[D]$ are causal (or of zero length). Hence,  $V \in \bfb{\g[D]}$, so that $\mathbf{Z} \cap \bfb{\g[D]} \neq \emptyset$.
\end{proof}
The final lemma needed to prove Theorem~\ref{theorem:gac-pdag} is Lemma~\ref{lemma:eqb-pdag}. This lemma relies on Lemma~\ref{lemma:2-pdag}, which depends on Lemma~\ref{lemma:non-causal-path-pdag}, \ref{lemma:1-pdag} and \ref{lemma:2-pdag}. We first give Lemma~\ref{lemma:eqb-pdag} with its proof. This is followed by Lemmas~\ref{lemma:non-causal-path-pdag}, \ref{lemma:1-pdag} and \ref{lemma:2-pdag} with their proofs. 

\begin{lemma}
    Let $\mathbf{X}$ and $\mathbf{Y}$ be disjoint node sets in a \mpdag{} $\g$. If $\g$ is b-amenable relative to $(\mathbf{X},\mathbf{Y})$ and $\mathbf{Z}$ satisfies \bforb{} relative to $(\mathbf{X},\mathbf{Y})$ in $\g$, then the following statements are equivalent:
   \begin{enumerate}[label = (\roman*)]
\item\label{l:eqb1-pdag} $\mathbf{Z}$ satisfies \bblck{} (see Definition~\ref{def:gac-pdag}) relative to $(\mathbf{X},\mathbf{Y})$ in~$\g$.
\item\label{l:eqb2-pdag}  $\mathbf{Z}$ satisfies \bblck{} relative to $(\mathbf{X},\mathbf{Y})$ in every $\DAG$ in $[\g]$.
\item\label{l:eqb3-pdag} $\mathbf{Z}$ satisfies \bblck{} relative to $(\mathbf{X},\mathbf{Y})$ in a $\DAG$ $\g[D]$ in $[\g]$.
\end{enumerate}
   \label{lemma:eqb-pdag}
\end{lemma}
\begin{proofof}[Lemma~\ref{lemma:eqb-pdag}]
This lemma is related to Lemma~10 from \cite{perkovic16},but instead of using Lemma 52 from \cite{perkovic16}, we use Lemma~\ref{lemma:2-pdag}.

   To prove $\neg$~\ref{l:eqb1-pdag}$\Rightarrow\neg$~\ref{l:eqb3-pdag} let $p$ be a proper b-non-causal definite status path from $\mathbf{X}$ to $\mathbf{Y}$ that is d-connecting given $\mathbf{Z}$ in~$\g$. The path corresponding to $p$ in any $\DAG$ $\g[D]$ in $[\g]$ is proper, non-causal (Lemma~\ref{lemma:posdir-path}) and d-connecting given $\mathbf{Z}$. 
   
The implication $\neg$~\ref{l:eqb3-pdag}$\Rightarrow\neg$~\ref{l:eqb2-pdag} trivially holds, so it is only left to prove that $\neg$~\ref{l:eqb2-pdag}$\Rightarrow\neg$~\ref{l:eqb1-pdag}.
   Thus, assume there is a $\DAG$ $\g[D]$ in $[\g]$ such that a proper b-non-causal path from $\mathbf{X}$ to $\mathbf{Y}$ in~$\g[D]$ is d-connecting given $\mathbf{Z}$.
  Among the shortest proper non-causal paths from $\mathbf{X}$ to $\mathbf{Y}$ that are d-connecting given $\mathbf{Z}$ in~$\g[D]$, choose a path $p$ with a minimal $\distancefrom{\mathbf{Z}}$ (Definition~\ref{def:distance from Z}). Let $\pstar$ in $\g$ be the path corresponding to $p$ in $\g[D]$. By Lemma~\ref{lemma:2-pdag}, $\pstar$ is a proper b-non-causal definite status path from $\mathbf{X}$ to $\mathbf{Y}$ that is d-connecting given $\mathbf{Z}$. 
\end{proofof}
\begin{lemma}
Let $\mathbf{X,Y}$ and $\mathbf{Z}$ be pairwise disjoint node sets in a \mpdag{} $\g$. Let $\mathbf{Z}$ satisfy \bamen{} and \bforb{} relative to $(\mathbf{X},\mathbf{Y})$ in~$\g$. Let $\g[D]$ be a $\DAG$ in $[\g]$ and let $p = \langle X = V_{0}, V_{1},\dots ,V_{k} = Y \rangle, k \ge 1,$ be a proper non-causal path from $X \in \mathbf{X}$ to $Y \in \mathbf{Y}$ that is d-connecting given $\mathbf{Z}$ in~$\g[D]$. Let $\pstar$ in~$\g$ denote the path corresponding to $p$. Then:
\begin{enumerate}[label=(\roman*)]
\item\label{l:ncp:pdag} Let $i,j \in \mathbb{N}, 0 < i < j \leq k,$ such that there is an edge $\langle V_i, V_j \rangle$ in~$\g$. The path $ \pstar(X,V_i) \oplus \langle V_i, V_j \rangle \oplus \pstar(V_j,Y)$ ($\pstar(V_j, Y)$ is possibly of zero length) is a proper b-non-causal path in~$\g$. For $j=i+1$, this implies that $\pstar$ is a proper b-non-causal path.
\item\label{l:ncp:b-pdag-new} If $X \leftarrow V_1$ and $V_1 \rightarrow V_2$ are not in $\g$ and there is an edge $\langle X,V_2 \rangle$ in $\g$, then $\pstar[q]=\langle X,V_2 \rangle \oplus \pstar(V_2,Y)$, ($\pstar(V_2, Y)$ is possibly of zero length) is a proper b-non-causal path in~$\g$.
\end{enumerate}
\label{lemma:non-causal-path-pdag}
\end{lemma}
\begin{proof}
This lemma is related to Lemma 50 from \cite{perkovic16}. In particular, \ref{l:ncp:pdag} in Lemma~\ref{lemma:non-causal-path-pdag} and (i) in Lemma 50 from \cite{perkovic16} and their proofs match. The result in \ref{l:ncp:b-pdag-new} differs in both statement and proof from (ii)-(iii) in Lemma 50 from \cite{perkovic16}.

All paths considered are proper as they are subsequences of $\pstar$, which corresponds to $p$.

\ref{l:ncp:pdag} We use proof by contradiction. Thus, suppose that $\pstar[q] = \pstar(X,V_i) \oplus \langle V_i, V_j \rangle \oplus \pstar(V_j,Y) $ is b-possibly causal in~$\g$. Then  $\{V_1,\dots,V_i,V_j,\dots,V_k\} \subseteq \bfb{\g}$. Since $\g$ is b-amenable relative to $(\mathbf{X,Y})$, $\pstar[q]$ and $\pstar(X,V_i)$ as well must start with $X \rightarrow V_1$.
Then $p$ also starts with $X \rightarrow V_1$ and since $p$ is non-causal, there is at least one collider on $p$.
Let $V_r, r \ge 1,$ be the collider closest to $X$ on $p$, then $V_r \in \bfb{\g[D]}$. Since $p$ is d-connecting given $\mathbf{Z}$, $\mathbf{Z} \cap \De(V_r,\g[D]) \neq \emptyset$. Since $\De(V_r,\g[D]) \subseteq \bfb{\g[D]}$ and since $\bfb{\g[D]} \subseteq \bfb{\g}$, this contradicts $\mathbf{Z} \cap \bfb{\g} = \emptyset$.

\ref{l:ncp:b-pdag-new} We again use proof by contradiction. Suppose neither $X \leftarrow V_1$ nor $V_1 \rightarrow V_2$ are in $\g$ and $\pstar[q]$ is a b-possibly causal path. Since $\g$ is b-amenable relative to $(\mathbf{X,Y})$, $X \to V_2$ is in~$\g$ and  $\{V_2,\dots,V_k\} \subseteq \bfb{\g}$.
Since $V_1 \rightarrow V_2$ is not in $\g$, either $V_1 - V_2$ or $V_1\leftarrow V_2$ is in $\g$. Hence, $V_1 \in \bPossDe(V_2,\g)$. Since $\bPossDe(V_2,\g) \subseteq \bfb{\g}$, it follows that $V_1 \in \bfb{\g}$.

Suppose that $V_1 \leftarrow V_2$ is in $\g$. Then $X \rightarrow V_2 \rightarrow V_1$ and rule $R2$ implies that $X \rightarrow V_1$ is in $\g$.
 Then $X \rightarrow V_1 \leftarrow V_2$, so since $p$ is d-connecting given $\mathbf{Z}$, $\mathbf{Z} \cap \De(V_1,\g[D]) \neq \emptyset$. But $\De(V_1,\g[D]) \subseteq \bPossDe(V_1,\g) \subseteq \bfb{\g}$, which contradicts that $\mathbf{Z} \cap \bfb{\g} = \emptyset$.

Otherwise, $V_1 -V_2$ is in $\g$. Additionally, $X \rightarrow V_2$ is in $\g$ and $\pstar(V_2,Y)$ is b-possibly causal. So since $\pstar$ is b-non-causal (from \ref{l:ncp:pdag}), $X \leftarrow V_1$ must be in $\g$. This contradicts our assumption in \ref{l:ncp:b-pdag-new}. 
\end{proof}

\begin{lemma}
Let $\mathbf{X,Y}$ and $\mathbf{Z}$ be pairwise disjoint node sets in a \mpdag{} $\g$. Let $\mathbf{Z}$ satisfy \bamen{} and \bforb{} relative to $(\mathbf{X},\mathbf{Y})$ in~$\g$. Let $\g[D]$ be a $\DAG$ in $[\g]$ and let $p$ be a shortest proper non-causal path from $\mathbf{X}$ to $\mathbf{Y}$ that is d-connecting given $\mathbf{Z}$ in~$\g[D]$. Let $\pstar$ in~$\g$ be corresponding path to $p$ in $\g[D]$. Then $\pstar$ is a proper b-non-causal definite status path in~$\g$ such that for every subpath $C_l \rightarrow C \leftarrow C_r$ of $\pstar$ there is no edge $\langle C_l,C_r\rangle$ in $\g$.
   \label{lemma:1-pdag}
\end{lemma}

\begin{proof}
This lemma is related to Lemma 51 from \cite{perkovic16}. 
Our lemma additionally contains the result that if $\langle C_l,C,C_r \rangle$ is a subpath of $\pstar$, then there is no edge $\langle C_l,C_r \rangle $ in $\g$. The proofs of this lemma and Lemma 51 from \cite{perkovic16} overlap for cases \ref{l:11}-\ref{l:13} and then diverge after that. 

Path $\pstar$ is proper and b-non-causal (\ref{l:ncp:pdag} in Lemma~\ref{lemma:non-causal-path-pdag}) in $\g$, so it is only left to prove that it is of definite status and that for any subpath $C_l \rightarrow C \leftarrow C_r$ of $\pstar$ there is no edge $\langle C_l,C_r\rangle$ in $\g$.
Let  $\pstar = \langle X = V_{0}, V_{1},\dots ,V_{k} = Y \rangle, k > 1, X \in \mathbf{X}, Y \in \mathbf{Y}$. We first prove that $\pstar$ is of definite status, by contradiction.

Hence, suppose that a node on $\pstar$ is not of definite status. Let $V_i, i\ge 1,$ be the node closest to $X$ on $\pstar$ that is not of definite status.
Then  $\langle V_{i-1}, V_{i}, V_{i+1} \rangle$ is shielded in $\g$ and there is an edge between $V_{i-1}$ and $V_{i+1}$ in $\g$. Let $q = p(X,V_{i-1}) \oplus \langle V_{i-1} , V_{i+1} \rangle \oplus p(V_{i+1},Y)$ in $\g[D]$. Let $\pstar[q]$ be the path corresponding to $q$ in $\g$. Then $\pstar[q]$ is proper and b-non-causal (Lemma~\ref{lemma:non-causal-path-pdag}) in $\g$. Hence, $q$ is also a proper non-causal path (Lemma~\ref{lemma:posdir-path}). Since $p$ is a shortest proper b-non-causal path from $X$ to $Y$ that is d-connecting given $\mathbf{Z}$, it follows that $q$ must be blocked by $\mathbf{Z}$.
The collider/non-collider status of all nodes, except possibly $V_{i-1}$ and $V_{i+1}$, is the same on $p$ and $q$. Hence, $V_{i-1}$ or $V_{i+1}$ block $q$, so $V_{i-1} \neq X$ or $V_{i+1} \neq Y$.
We now discuss the different cases for the collider/non-collider status of $V_{i-1}$ and $V_{i+1}$ on $p$ and $q$ and derive a contradiction in each case.

\begin{enumerate}[label=(\arabic*)]
\item\label{l:11}  $V_{i-1}$ is a non-collider on $p$, a collider on $q$ and $\De(V_{i-1},\g[D]) \cap \mathbf{Z} = \emptyset$.
Then $V_{i+1} \rightarrow V_{i-1} \rightarrow V_{i}$  and rule $R2$ implies that $V_{i} \leftarrow V_{i+1}$ is in $\g[D]$. Since $p$ is d-connecting given $\mathbf{Z}$, $\De(V_{i},\g[D]) \cap \mathbf{Z} \neq \emptyset$. As $\De(V_{i},\g[D]) \subseteq \De(V_{i-1},\g[D])$, this contradicts $\De(V_{k-1},\g[D]) \cap \mathbf{Z} = \emptyset$.

\item $V_{i+1}$ is a non-collider on $p$, a collider on $q$ and $\De(V_{i+1},\g[D]) \cap \mathbf{Z} = \emptyset$.
This case is symmetric to case \ref{l:11} and the same argument leads to a contradiction.

\item\label{l:13} $V_{i-1}$ is a collider on $p$, a non-collider on $q$ and $V_{i-1} \in \mathbf{Z}$.
Since $V_{i-1}$ is of definite status on $\pstar$ it follows that $V_{i-2} \rightarrow V_{i-1} \leftarrow V_{i}$ is in $\g$. This implies $V_{i}$ is a definite non-collider on $\pstar$, which is a contradiction.

\item\label{l:14} $V_{i+1}$ is a collider on $p$, a non-collider on $q$ and $V_{i+1} \in \mathbf{Z}$. Then $V_{i} \rightarrow V_{i+1} \leftarrow V_{i+2}$ and $V_{i-1} \leftarrow V_{i+1}$ is in $\g[D]$. $V_{i+1}$ is not of definite status on $\pstar$, otherwise $V_{i}$ would be of definite status on $\pstar$. 
Thus, there is an edge $\langle V_{i},V_{i+2} \rangle$ in $\g$. The path $r = p(X,V_{i}) \oplus \langle V_{i} , V_{i+2} \rangle \oplus p(V_{i+2},Y)$ is proper, b-non-causal (Lemma~\ref{lemma:non-causal-path-pdag}) and shorter than $p$. Hence, $r$ must be blocked by $\mathbf{Z}$ in $\g[D]$.  
The collider/non-collider status of all nodes except possibly $V_{i}$ and $V_{i+2}$ is the same on $r$ and $p$, so either $V_{i}$ or $V_{i+2}$ must block $r$.

Since $V_{i} \rightarrow V_{i+1} \rightarrow V_{i-1}$ is in $\g$, rule $R2$ implies that $V_{i-1} \leftarrow V_{i}$ is in $\g$. Since $V_{i-1} \leftarrow V_{i}$ is on both $r$ and $p$, $V_{i}$ cannot block $r$. 
Additionally, $V_{i+1} \leftarrow V_{i+2}$ is in $\g[D]$ so $V_{i+2}$ is a non-collider on $p$. Thus, $V_{i+2}$ must be a collider on $r$ and $\De(V_{i+2},\g[D]) \cap \mathbf{Z} = \emptyset$. However, by assumption in \ref{l:14}, $V_{i+1} \in \mathbf{Z}$ and $V_{i+1} \in \De(V_{i+2},\g[D])$.
\end{enumerate}

Lastly, let $C_l \rightarrow C \leftarrow C_r$ be a subpath of $\pstar$ and suppose for a contradiction that there is an edge $\langle C_l,C_r\rangle$ in $\g$. Let $\pstar[q] = \pstar(X,C_l)\oplus \langle C_l,C_r\rangle \oplus \pstar(C_r,Y)$. Then $\pstar[q]$ is proper and b-non-causal ( Lemma~\ref{lemma:non-causal-path-pdag}) in $\g[D]$. Let $q$ in $\g[D]$ be the path corresponding to $\pstar[q]$ in $\g$. Then $q$ is a proper non-causal path from $X$ to $Y$ that is shorter than $p$, so $q$ must be blocked by $\mathbf{Z}$. Then as above, either $C_l$ or $C_r$ must block $q$. 

Since $C_l$ ($C_r$) is a non-collider on $p$, it must be a collider on $q$ and $\De(C_l,\g[D]) \cap \mathbf{Z} \neq \emptyset$ ($\De(C_r,\g[D]) \cap \mathbf{Z} \neq \emptyset$). Since $p$ is d-connecting given $\mathbf{Z}$, $\De(C,\g[D]) \cap \mathbf{Z} \neq \emptyset$. Additionally, $\De(C_l,\g[D]) \supseteq \De(C,\g[D])$ ($\De(C_r,\g[D]) \supseteq \De(C,\g[D])$), which contradicts $\De(C_l,\g[D]) \cap \mathbf{Z} \neq \emptyset$ $(\De(C_r,\g[D]) \cap \mathbf{Z} \neq \emptyset$).
\end{proof}

\begin{lemma}
 Let $\mathbf{X,Y}$ and $\mathbf{Z}$ be pairwise disjoint node sets in a \mpdag{} $\g$. Let $\mathbf{Z}$ satisfy \bamen{} and \bforb{} relative to $(\mathbf{X},\mathbf{Y})$ in~$\g$. Let $\g[D]$ be a $\DAG$ in $[\g]$ and let $p$ be a path with minimal $\distancefrom{\mathbf{Z}}$ among the shortest proper non-causal paths from $\mathbf{X}$ to $\mathbf{Y}$ that are d-connecting given $\mathbf{Z}$ in~$\g[D]$. Let $\pstar$ in $\g$ be the path corresponding to $p$ in $\g[D]$. Then $\pstar$ is a proper b-non-causal definite status path from $\mathbf{X}$ to $\mathbf{Y}$ that is d-connecting given $\mathbf{Z}$ in~$\g$.
   \label{lemma:2-pdag}
\end{lemma}

\begin{proof}
This lemma is related to Lemma~52 from \cite{perkovic16}. The line of reasoning used in this first part of this proof overlaps with the proof of Lemma~52. We will point out where the two proofs diverge.

From Lemma~\ref{lemma:1-pdag} we know that $\pstar$ is a proper b-non-causal definite status path from $\mathbf{X}$ to $\mathbf{Y}$ in $\g$.
We only need to prove that it is also d-connecting given $\mathbf{Z}$ in $\g$. 

Since $p$ is d-connecting given $\mathbf{Z}$ in $\g[D]$ and $\pstar$ is of definite status, it follows that no definite non-collider on $\pstar$ is in $\mathbf{Z}$ and that every collider on $\pstar$ has a possible descendant in $\mathbf{Z}$ (Lemma~\ref{lemma:non-causal-path-pdag}). Since every collider on $\pstar$ has a b-possibly causal path to $\mathbf{Z}$, by Lemma~\ref{lemma:unshielded-analog} there is b-possibly causal definite status path from every collider on $\pstar$ to a node in $\mathbf{Z}$. Let $C$ be an arbitrary collider on $\pstar$ and let $\pstar[d]$ be a shortest b-possibly causal definite status path from $C$ to a node in $\mathbf{Z}$. It is only left to show that $\pstar[d]$ is causal in $\g$, since then $C \in \An(\mathbf{Z},\g)$.

If $\pstar[d]$ starts with a directed edge out of $C$, then $\pstar[d]$ is causal in $\g$ (Lemma~\ref{lemma:unshielded-edgeanalog}). Otherwise, $\pstar[d] = C - S \dots Z, Z \in \mathbf{Z}$ (possibly $S=Z$). We will prove that this leads to a contradiction.
Hence, let $ C_{l} \rightarrow C \leftarrow C_{r}$ be a subpath of $\pstar$. 

From this point onwards, this proof deviates somewhat from the proof of Lemma~52 from \cite{perkovic16}, due to the additional result in Lemma~\ref{lemma:1-pdag}.
Since $C_l \rightarrow C - S$ ($C_r \rightarrow C - S$) is in $\g$, rule $R1$ and $R2$ imply that either $C_l \rightarrow S$ or $C_l -S$ ($C_r \rightarrow S$ or $C_r -S$) is in $\g$. Suppose that $C_l - S$ is in $\g$. Since $C_l \notin \Adj(C_r,\g)$ (Lemma~\ref{lemma:1-pdag}), $C_r - S$ must be in $\g$, otherwise $C_l - S \leftarrow C_r$ violates $R1$ in $\g$. But then $C_l \rightarrow C \leftarrow C_r$, $C_l -S -C_r$ and $C -S$ violate $R3$ in $\g$. 

Hence, $C_l \rightarrow S$ is in $\g$. Then $C_r \rightarrow S$ must be in $\g$, otherwise $C_l \notin \Adj(C_r,\g)$ and $C_l \rightarrow S - C_r$ violates $R1$. Now, depending on whether $S$ is a node on $p$, we can derive the final contradiction.

Suppose $S$ is not on $p$. Then if $S \notin \mathbf{X} \cup \mathbf{Y}$, $p(X,C_l) \oplus \langle C_l ,S,C_r\rangle \oplus p(C_r,Y)$ is a proper non-causal path from $\mathbf{X}$ to $\mathbf{Y}$ in $\g[D]$ that is of the same length as $p$, but with a shorter distance-from-$\mathbf{Z}$ than $p$ and d-connecting given $\mathbf{Z}$. This contradicts our choice of $p$. Otherwise, suppose $S \in \mathbf{X}$. Then $\langle S ,C_r \rangle \oplus p(C_r, Y)$ contradicts our choice of $p$. Otherwise, $S \in \mathbf{Y}$. Then $S \notin \bfb{\g}$ otherwise, $Z \in \bfb{\g}$ since $Z \in \bPossDe(S,\g)$. Since $S \notin \bfb{\g}$, it follows that $S \notin \bfb{\g[D]}$, so $p(X,C_r) \oplus \langle C_l,S \rangle$ is a non-causal path from $\mathbf{X}$ to $\mathbf{Y}$ in $\g[D]$. Then $p(X,C_r) \oplus \langle C_l,S \rangle$ contradicts our choice of $p$.

Otherwise, $S$ must be on $p$. Hence, $S \notin \mathbf{X}$. Suppose first that $S$ is on $p(X,C_r)$. Let $q= p(X,S) \oplus \langle S,C_r\rangle \oplus p(C_r, Y)$. Since $q$ is proper, non-causal and shorter than $p$, we only need to prove that $q$ is d-connecting given $\mathbf{Z}$ to derive a contradiction. For this we only need to discuss the collider/non-collider status of $S$ on $q$. If $S$ is a collider on $q$, then $q$ is d-connecting given $\mathbf{Z}$. Otherwise, $S$ is a non-collider on $q$. Then since $S \leftarrow C_r$ is on $q$, $S$ must also be a non-collider on $p$. Since $p$ is d-connecting given $\mathbf{Z}$, $S \notin \mathbf{Z}$. Thus, $q$ must be d-connecting given $\mathbf{Z}$ in $\g$. 

Otherwise, $S$ is on $p(C_r,Y)$. Then let $r = p(X,C_l) \oplus \langle C_l,S \rangle \oplus p(S,Y)$. Since $S \notin \bfb{\g}$ (otherwise $Z \in \bfb{\g}$ since $Z \in \bPossDe(S,\g)$) and since $r$ is proper, it follows that $r$ is a non-causal path. Hence, we only need to prove that $r$ is d-connecting given $\mathbf{Z}$ to derive a contradiction.  For this we again only discuss the collider/non-collider status of $S$ on $q$. If $S$ is a collider on $r$, $r$ is d-connecting given $\mathbf{Z}$. Otherwise, $S$ is a non-collider on $r$. Then since $C_l \rightarrow S$ is on $q$, $S$ must also be a non-collider on $p$. Since $p$ is d-connecting given $\mathbf{Z}$, $S \notin \mathbf{Z}$. Thus, $r$ must be d-connecting given $\mathbf{Z}$ in $\g$. 
\end{proof}

\subsection{PROOF OF THEOREM~\ref{theorem:constructive-set-mpdag}} \label{subsec:proof-constr}
\begin{proofof}[Theorem~\ref{theorem:constructive-set-mpdag}]
This theorem is related to Theorem~14 from \cite{perkovic16}. This proof relies on similar line of reasoning however since Theorem~14 from \cite{perkovic16} states a somewhat different result and relies on a few lemmas for the proof, we do not make a direct comparison between the two as we did with the result presented in Section~\ref{sec:proofadjust-mpdag}.

We only need to prove that if there is a set that satisfies the b-adjustment criterion relative to $(\mathbf{X,Y})$ in $\g$, then $\badjustb{\g}$ also satisfies the b-adjustment criterion relative to $(\mathbf{X,Y})$ in $\g$.
Hence, assume that $\mathbf{Z}$ satisfies the b-adjustment criterion relative to $(\mathbf{X,Y})$ in $\g$ and that $\badjustb{\g}$ does not satisfy the b-adjustment criterion relative to $(\mathbf{X,Y})$ in $\g$. We will show that this leads to a contradiction.

Since $\mathbf{Z}$ is an adjustment set relative to $(\mathbf{X,Y})$ in $\g$ (Theorem~\ref{theorem:gac-pdag}), $\g$ is b-amenable relative to $(\mathbf{X,Y})$. By construction $\badjustb{\g}$ satisfies \bforb{}, so it must violate \bblck{}.
Let $p =\langle X = V_{0}, V_{1},\dots ,V_{k} = Y \rangle, k \ge 1, X \in \mathbf{X}, Y \in \mathbf{Y}$ be a shortest proper b-non-causal definite status paths from $\mathbf{X}$ to $\mathbf{Y}$ in $\g$ that is d-connecting given $\badjustb{\g}$. Since $\mathbf{Z}$ blocks $p$, $k >1$. So there is at least one non-endpoint node on $p$. 

Since $p$ is d-connecting given $\badjustb{\g}$ and $\badjustb{\g} = \bPossAn(\mathbf{X} \cup \mathbf{Y}, \g)\setminus(\mathbf{X} \cup \mathbf{Y} \cup \bfb{\g})$ any collider on $p$ is in $\bPossAn(\mathbf{X} \cup \mathbf{Y}, \g)$. Additionally, no collider $C$ on $p$ is in $\bfb{\g}$ otherwise, $\De(C,\g) \subseteq \bfb{\g}$ and $\bfb{\g} \cap \badjustb{\g} = \emptyset$ contradicts that $p$ is d-connecting given $\badjustb{\g}$. Thus, every collider on $p$ is in $\bPossAn(\mathbf{X} \cup \mathbf{Y}, \g) \setminus \bfb{\g}$.

Any definite non-collider on $p$ is a b-possible ancestor of $X$, $Y$ or a collider on $p$. Hence, any definite non-collider on $p$ is in $\bPossAn(\mathbf{X} \cup \mathbf{Y},\g)$. Since $p$ is d-connecting given $\badjustb{\g}$, no definite non-collider on $p$ is in $\badjustb{\g}$. Then since  $\badjustb{\g} = \bPossAn(\mathbf{X} \cup \mathbf{Y}, \g)\setminus(\mathbf{X} \cup \mathbf{Y} \cup \bfb{\g})$, every definite non-collider on $p$ is in $\mathbf{X} \cup \mathbf{Y} \cup \bfb{\g}$.
Since $p$ is proper, no definite non-collider on $p$ is in $\mathbf{X}$. Additionally, if there was a definite non-collider $V$ on $p$ such that $V \in \mathbf{Y} \setminus \bfb{\g}$, then $p(X,V)$ would be a shorter proper b-non-causal definite status path in $\g$ that is d-connecting given $\badjustb{\g}$. Hence, any definite non-collider on $p$ must be in $\bfb{\g}$. 

Suppose that there is no collider on $p$. Then since $\mathbf{Z}$ blocks $p$, a non-collider on $p$ must be in $\mathbf{Z}$. This contradicts that $\mathbf{Z}$ satisfies \bforb{}.
Thus, there is a collider on $p$ so let $V_{i-1} \rightarrow V_{i} \leftarrow V_{i+1}$ be a subpath of $p$. If there is a definite non-collider on $p$, then $V_{i-1}$ or $V_{i+1}$ is a non-collider on $p$. Suppose without loss of generality that $V_{i-1}$ is a definite non-collider on $p$. Then $V_{i-1} \in \bfb{\g}$. Since $V_{i-1} \rightarrow V_i$ is in $\g$, $V_i \in \bfb{\g}$, which contradicts that $V_i$ is a collider on $p$.
Thus, there is no definite non-collider on $p$. 

Since there is at least one collider on $p$ and no definite non-collider is on $p$, it follows that $p$ is of the form $ X \rightarrow V_1 \leftarrow Y$ in $\g$. Since $V_1 \in \bPossAn(\mathbf{X} \cup \mathbf{Y},\g) \setminus \bfb{\g}$, let $q$ be a shortest b-possibly causal definite status path from $V_1$ to a node in $\mathbf{X} \cup \mathbf{Y}$. Then $q = \langle V_1, \dots , V\rangle$, for some $V \in \mathbf{X} \cup \mathbf{Y}$. Then $V \in \mathbf{X}$ otherwise, $V_1 \in \bfb{\g} $.

Thus, $r = (-p)(Y,V_1) \oplus q$ is a b-possibly causal path, so let $r'$ be an unshielded subsequence of $r$ that forms a b-possibly causal path from $Y$ to $\mathbf{X}$ in $\g$. Since $(-r)$ is a proper path with respect to $\mathbf{X}$, $(-r')$ is also a proper path with respect to $\mathbf{X}$. Then $(-r')$ must be a b-non-causal path otherwise, $Y \in \bfb{\g}$ which also implies $V_1 \in \bfb{\g}$. 
Since $r'$ is also unshielded, $(-r')$ is a proper b-non-causal definite status path from $\mathbf{X}$ to $\mathbf{Y}$ in $\g$. Thus, $\mathbf{Z}$ must block $(-r')$. Since $r'$ is b-possibly causal, $(-r')$ does not contain a collider, so a definite non-collider on $(-r')$ must be in $\mathbf{Z}$. However, all definite non-colliders on $(-r')$ are also on $q$, so $\mathbf{Z}$ cannot block both $(-r')$ and $p$ in $\g$.
\end{proofof}

\section{EMPIRICAL STUDY} \label{sec:study}

The following empirical study compares the runtimes of local IDA and our semi-local IDA on $\CPDAG$s.
We consider the $20\ 000$ simulation settings described in the paper. The times are recorded in seconds on an Intel(R) Core(TM) i7-4765T CPU \@ 2.00GHz processor running under Fedora 24 and using R version 3.4.0 and pcalg version 2.4-6.)
The summary is given in Table~\ref{table:sum}.
\begin{center}
\begin{tabular}{ |c|c|c|c| } 
 \hline
				& 	Median	& Mean		& Max \\ \hline 
 Local IDA 		& 	0.003	& 0.003		& 0.009 \\ \hline 
 Semi-local IDA &   0.003	& 0.016		& 4.881 \\ 
 \hline
\end{tabular}
\captionof{table}[tb]{Median, mean and max computation times of local IDA and semi-local IDA.}
 \label{table:sum}
\end{center}
We see that the median computation times are identical for both methods. The mean and maximum computation times, however, are larger for semi-local IDA. This indicates some outliers in the computation time of semi-local IDA. This could be explained by the presence of a small number of $\CPDAG$s where our semi-local method is forced to orient a large subgraph of the $\CPDAG$. 

We also investigated the difference in runtimes between semi-local IDA and local IDA as a function of the number of variables ($p$) and the expected neighborhood size ($E[N]$), these results are given in Table~\ref{table:meanrunp} and Table~\ref{table:meanrunen} respectively.
We see that the mean difference in runtimes increases with $p$, while there is not a very clear relationship with neighborhood size.

\begin{center}
\begin{tabular}{ |c|c| } 
 \hline
 $p$ 		& 	mean difference \\ \hline 
 20	  &	0.009 	\\ \hline
30	& 	0.007  \\	\hline
40	& 	0.010 \\\hline
50	& 	0.009 \\\hline
60	& 	0.016 	\\		\hline 
70	& 	0.013 \\\hline
80	&	0.014 	\\\hline
90	&	0.015 	\\\hline
100	& 	0.021\\
 \hline
\end{tabular}
\captionof{table}[tb]{The mean runtime difference aggregated according to the number of variables ($p$).}
\label{table:meanrunp}
\end{center}

\begin{center}
\begin{tabular}{ |c|c| } 
 \hline
 $E[N]$ 		& 	mean difference \\ \hline 
3	& 	0.016  \\	\hline
4	& 	0.013 \\\hline
5	& 	0.013 \\\hline
6	& 	0.011 	\\		\hline 
7	& 	0.009 \\\hline
8	&	0.011 	\\\hline
9	&	0.015 	\\\hline
10	& 	0.014\\
 \hline
\end{tabular}
\captionof{table}[tb]{The mean runtime difference aggregated according to the expected neighborhood size ($E[N]$).}
\label{table:meanrunen}
\end{center}

\vskip 0.2in
\bibliographystyle{apalike}
\bibliography{biblioteka}

\begin{thebibliography}{}

\bibitem[Ali et~al., 2005]{ali2012towards}
Ali, A.~R., Richardson, T.~S., Spirtes, P.~L., and Zhang, J. (2005).
\newblock Towards characterizing {M}arkov equivalence classes for directed
  acyclic graphs with latent variables.
\newblock In {\em Proceedings of UAI 2005}.

\bibitem[Chickering, 2002]{Chickering02-optimal}
Chickering, D.~M. (2002).
\newblock Optimal structure identification with greedy search.
\newblock {\em J. Mach. Learn. Res.}, 3:507--554.

\bibitem[Dor and Tarsi, 1992]{dorTarsi92}
Dor, D. and Tarsi, M. (1992).
\newblock A simple algorithm to construct a consistent extension of a partially
  oriented graph.
\newblock {\em Technicial Report R-185, Cognitive Systems Laboratory, UCLA}.

\bibitem[Eigenmann et~al., 2017]{eigenmann17}
Eigenmann, M., Nandy, P., and Maathuis, M.~H. (2017).
\newblock Structure learning of linear {G}aussian structural equation models
  with weak edges.
\newblock In {\em Proceedings of UAI 2017}.

\bibitem[Hauser and B{\"u}hlmann, 2012]{hauserBuehlmann12}
Hauser, A. and B{\"u}hlmann, P. (2012).
\newblock Characterization and greedy learning of interventional {M}arkov
  equivalence classes of directed acyclic graphs.
\newblock {\em J. Mach. Learn. Res.}, 13:2409--2464.

\bibitem[Hoyer et~al., 2008]{hoyer08}
Hoyer, P.~O., Hyvarinen, A., Scheines, R., Spirtes, P.~L., Ramsey, J., Lacerda,
  G., and Shimizu, S. (2008).
\newblock Causal discovery of linear acyclic models with arbitrary
  distributions.
\newblock In {\em Proceedings of UAI 2008}, pages 282--289.

\bibitem[Hyttinen et~al., 2015]{hyttinen2015calculus}
Hyttinen, A., Eberhardt, F., and J{\"a}rvisalo, M. (2015).
\newblock Do-calculus when the true graph is unknown.
\newblock In {\em Proceedings of UAI 2015}, pages 395--404.

\bibitem[Kalisch et~al., 2012]{kalischpcalg}
Kalisch, M., M\"achler, M., Colombo, D., Maathuis, M.~H., and B\"uhlmann, P.
  (2012).
\newblock Causal inference using graphical models with the {R} package {pcalg}.
\newblock {\em J. Stat. Softw.}, 47(11):1--26.

\bibitem[Maathuis and Colombo, 2015]{maathuis2013generalized}
Maathuis, M.~H. and Colombo, D. (2015).
\newblock A generalized back-door criterion.
\newblock {\em Ann. Stat.}, 43:1060--1088.

\bibitem[Maathuis et~al., 2010]{MaathuisColomboKalischBuehlmann10}
Maathuis, M.~H., Colombo, D., Kalisch, M., and B\"uhlmann, P. (2010).
\newblock Predicting causal effects in large-scale systems from observational
  data.
\newblock {\em Nat. Methods}, 7:247--248.

\bibitem[Maathuis et~al., 2009]{MaathuisKalischBuehlmann09}
Maathuis, M.~H., Kalisch, M., and B\"uhlmann, P. (2009).
\newblock Estimating high-dimensional intervention effects from observational
  data.
\newblock {\em Ann. Stat.}, 37:3133--3164.

\bibitem[Malinsky and Spirtes, 2017]{malinsky2017estimating}
Malinsky, D. and Spirtes, P. (2017).
\newblock Estimating bounds on causal effects in high-dimensional and possibly
  confounded systems.
\newblock {\em Int. J. of Approx. Reason.}

\bibitem[Meek, 1995]{meek1995causal}
Meek, C. (1995).
\newblock Causal inference and causal explanation with background knowledge.
\newblock In {\em Proceedings of UAI 1995}, pages 403--410.

\bibitem[Nandy et~al., 2017]{nandy2014estimating}
Nandy, P., Maathuis, M.~H., and Richardson, T.~S. (2017).
\newblock Estimating the effect of joint interventions from observational data
  in sparse high-dimensional settings.
\newblock {\em Ann. Stat.}, 45(2):647--674.

\bibitem[Pearl, 2009]{Pearl2009}
Pearl, J. (2009).
\newblock {\em Causality: Models, Reasoning, and Inference}.
\newblock Cambridge University Press, New York, NY, second edition.

\bibitem[Perkovi\'c et~al., 2015]{perkovic15_uai}
Perkovi\'c, E., Textor, J., Kalisch, M., and Maathuis, M.~H. (2015).
\newblock A complete generalized adjustment criterion.
\newblock In {\em Proceedings of UAI 2015}, pages 682--691.

\bibitem[Perkovi\'c et~al., 2018]{perkovic16}
Perkovi\'c, E., Textor, J., Kalisch, M., and Maathuis, M.~H. (2018).
\newblock Complete graphical characterization and ´ construction of adjustment
  sets in {M}arkov equivalence classes of ancestral graphs.
\newblock {\em J. Mach. Learn. Res.}, 18.

\bibitem[Richardson and Spirtes, 2002]{richardson2002ancestral}
Richardson, T.~S. and Spirtes, P. (2002).
\newblock Ancestral graph {M}arkov models.
\newblock {\em Ann. Stat.}, 30:962--1030.

\bibitem[Robins, 1986]{robins1986new}
Robins, J.~M. (1986).
\newblock A new approach to causal inference in mortality studies with a
  sustained exposure period-application to control of the healthy worker
  survivor effect.
\newblock {\em Math. Mod.}, 7:1393--1512.

\bibitem[Rothenh\"ausler et~al., 2018]{ernestroth2016}
Rothenh\"ausler, D., Ernest, J., and B\"uhlmann, P. (2018).
\newblock Causal inference in partially linear structural equation models:
  identifiability and estimation.
\newblock {\em Ann. Stat.}
\newblock To appear.

\bibitem[Scheines et~al., 1998]{tetrad1998}
Scheines, R., Spirtes, P., Glymour, C., Meek, C., and Richardson, T. (1998).
\newblock The {TETRAD} project: constraint based aids to causal model
  specification.
\newblock {\em Multivar. Behav. Res.}, 33(1):65--117.

\bibitem[Shpitser, 2012]{shpitser2012avalidity}
Shpitser, I. (2012).
\newblock Appendum to ``{O}n the validity of covariate adjustment for
  estimating causal effects''.
\newblock Personal communication.

\bibitem[Shpitser et~al., 2010]{shpitser2012validity}
Shpitser, I., VanderWeele, T., and Robins, J.~M. (2010).
\newblock On the validity of covariate adjustment for estimating causal
  effects.
\newblock In {\em Proceedings of UAI 2010}, pages 527--536.

\bibitem[Spirtes et~al., 2000]{spirtes2000causation}
Spirtes, P., Glymour, C., and Scheines, R. (2000).
\newblock {\em Causation, Prediction, and Search}.
\newblock MIT Press, Cambridge, MA, second edition.

\bibitem[van~der Zander and Li\'skiewicz, 2016]{vanDerZander16}
van~der Zander, B. and Li\'skiewicz, M. (2016).
\newblock Separators and adjustment sets in {M}arkov equivalent {DAG}s.
\newblock In {\em Proceedings of AAAI 2016}, pages 3315--3321.

\bibitem[van~der Zander et~al., 2014]{vanconstructing}
van~der Zander, B., Li\'skiewicz, M., and Textor, J. (2014).
\newblock Constructing separators and adjustment sets in ancestral graphs.
\newblock In {\em Proceedings of UAI 2014}, pages 907--916.

\bibitem[Wang et~al., 2017]{wang2017permutation}
Wang, Y., Solus, L., Yang, K.~D., and Uhler, C. (2017).
\newblock Permutation-based causal inference algorithms with interventions.
\newblock In {\em Proceedings of NIPS 2017}, pages 5824--5833.

\bibitem[Zhang, 2008]{zhang2008completeness}
Zhang, J. (2008).
\newblock On the completeness of orientation rules for causal discovery in the
  presence of latent confounders and selection bias.
\newblock {\em Artif. Intell.}, 172:1873--1896.

\end{thebibliography}

\end{document}